\documentclass[11pt,reqno]{amsart}
\usepackage{color}
\DeclareMathOperator{\A}{\mathcal A}
\newtheorem{theorem}{Theorem}[section]
\newtheorem{lemma}[theorem]{Lemma}
\newtheorem{corollary}[theorem]{Corollary} 

\theoremstyle{definition}
\newtheorem{definition}[theorem]{Definition}
\newtheorem{example}[theorem]{Example}
\newtheorem{proposition}[theorem]{Proposition}
\newtheorem{question}[theorem]{Question}

\theoremstyle{remark}
\newtheorem{remark}[theorem]{Remark} 
\usepackage{amsmath, amssymb}
\usepackage{mathrsfs}

\begin{document}

\title{topological extensions of elementary measurable spaces}

\author{Nidaa Hasan Haji}
\address{Department of Mathematics and computer applications, College of Science, Al-Muthanna University, Iraq}
\email{nedaa.hassan@mu.edu.iq}

\author{Abdolaziz Hesari}
\address{Department of Mathematics, Shahid Chamran University of ahvaz, Ahvaz, Iran}
 \email{ahsaeiaziz@gmail.com}

\begin{abstract}
An elementary space is a pair $(X, \A)$, where $\A$ is an algebra of subsets of $X$. Every  elementary space $(X, \A)$ is naturally  a topological space when equipped with the topology $\tau_{\A}$ generated by $\A$.
 Elementary mappings on elementary spaces are defined in a manner analogous to measurable mappings on measurable spaces. In addition, the notion of an {\it  extension} of an elementary space is introduced.  Let $M(X, \A)$, denote the ring of all real-valued elementary functions on  $(X, \A)$.  In this article, we study the ring $M(X, \A)$ and  extensions  of an elementary space. We show that, for any elementary space $(X, \A)$,
 there exists a unique compact elementary extension  $(\beta_{\A}X, \bar\A)$ of $X$, analogous to the Stone-$\check C$ech compactification,  which contains every other extension of $X$. This extension has the property that every elementary map from $X$ into any compact topological space extends continuously to $\beta_{\A}X$.
 Next, we show that $\beta_{\A}X$ contains a realcompact extension $(\upsilon_{\A}X, \bar\A_R)$ of $X$, such that every elementary map from $X$ into any realcompact topological space extends continuously to $\upsilon_{\A}X$. Moreover, $\upsilon_{\A}X$ is the smallest realcompact  extension of $X$ and $(X, \A)$ is realcompact if and only if $X=\upsilon_{\A}X$. Also, we   see that, every metric space equipped with the $\sigma$-algebra of all Borel sets is realcompact.
 Finally, we observe that the maximal ring of quotients of $M(X, \A)$  is the direct limit of locally constant (continuous) functions on dense open subsets of $(X, \tau_{\A})$. Some questions raised in the literature are also answered.

 \end{abstract}
\maketitle
% section1 --------   Introduction 
\section{Introduction} 
An algebra in a nonempty set $X$, is a nonempty subfamily $\A$ of subsets of $X$ which  is closed under   finite unions and complementation. In this case,   the pair $(X, \A)$ is called  an {\it elementary measurable} space or simply elementary space and the members of $\A$ are called {\it elementary} sets. Moreover, if the family $\A$ is closed under countable unions, it is called a $\sigma$-algebra, the pair $(X, \A)$,  a measurable space and the members of $\A$ are measurable sets. A mapping $f $ from  a  measurable space $(X, \A)$ into  a topological space $Y$ is called measurable  if $f^{-1}(U)\in \A$, for every open subset $U$ of $Y$.  A mapping such that  is called an {\it elementary} map, whenever $(X, \A)$ is only an elementary space. In either case, in this paper, these  mappings are referred to as $\A$-measurable, to indicate the algebra or $\sigma$-algebra $\A$ in the set  $X$. Whenever $(X, \A)$ is a measurable space,  the set of all real-valued $\A$-measurable functions on $X$ is denoted by $M(X, \A)$. It is a   well-known result that, this set  is a commutative lattice-ordered ring with unit. 

The ring of all real-valued $\A$-measurable functions on a measurable space have been studied by many authors. The first paper concerning this ring dates back to 1966 \cite{G}. In a recent work, for a measurable space $(X, \A)$,  Estaji et al. in \cite{EMY}, use the set of all prime ideals of $\A$ (as a Boolean algebra) to characterize maximal ideals of $M(X, \A)$ and  prove a measure-theoretic analogue of the Gelfand-Kolomogroff theorem. In \cite{ABS}, the authors, initiate a duality between the maximal ideals of $M(X, \A)$ and $\A$-ultrafilters (filters  consisting only of measurable sets) on $X$. Using this duality, it has been shown that the structure space of $M(X, \A)$ (regarded with the hull-kernel topology) is isomorphic to the set $\hat X$, consisting of all $\A$-ultrafilters on $X$, equipped with the Stone topology. Then, they investigate the quotient rings of $M(X, \A)$, where $(X, \A)$ is a measurable space, and prove that each ideal of  $M(X, \A)$ is absolutely convex. Using this, maximal ideals of $M(X, \A)$ characterized  as those ideals $I$ where $M(X, \A)/I$ is totally ordered, see \cite[Theorem 3.5]{ABS}. Thus, a maximal ideal in $M(X, \A)$ is called real, provided that $M(X, \A)/M\simeq\mathbb R$. In addition, real maximal ideals of $M(X, \A)$ has been characterized in terms of their associated $\A$-ultrafilters on $X$, see \cite[Theorem 3.9]{ABS}.

 Acharyya et al. in \cite{AABS}, which is a continuation of \cite{ABS}, defined realcompact measurable spaces as those measurable spaces $(X, \A)$ in which every real maximal ideal of $M(X, \A)$ is fixed. They also, show that $M(X, \A)$ can be embedded monomorphically in $C(RMax(X, \A))$, where $RMax(X, \A)$ is the subspace of the structure space of $M(X, \A)$ consisting of all real maximal ideals of $M(X, \A)$.

In this paper, we investigate elementary spaces and the ring of  real-valued elementary functions on such spaces. Specifically, we equip a set $X$ with an algebra $\A$, rather than the stronger structure of a $\sigma$-algebra, and observe that for many of practical applications, it is sufficient for  $\A$ to be closed under finite unions  and intersections. However, if one wishes to take limits of   countable families of elementary functions, then  $\A$ must necessarily  be a $\sigma$-algebra.

 We develop the theory of the ring of  elementary functions  in a manner analogous to the theory of the ring $C(Y)$, the ring of all continuous real-valued functions on a topological space $Y$. To  this end, we introduce a fundamental tool: for an elementary space $(X, \A)$, we equipe the set $X$ by a topology  $\tau_{\A}$ generated by $\A$, which establishes  a connection between $M(X, \A)$ and $C(X)$ and plays a central role throughout the paper. As in the classical theory of topological extensions, we make the following definition: an elementary space $(T, \A')$ is said to be an elementary extension of  $(X, \A)$ provided  that $X$ is a dense subspace of $T$ with respect to the topology $\tau_{\A'}$ and  $\A=\{A'\cap X : A'\in \A'\}$. This leads naturally to the study of compact and realcompact extensions of an elementary space. In turn, this enables us to characterize compactness and realcompactness and to develop the theory of extending mappings on elementary spaces.

Here is an outline of the structure of the paper.
After some preliminaries in section 2, that are needed in the sequel, in section 3, we investigate elementary extensions of  a given elementary space $(X, \A)$ and show that they  all lie in a unique maximum elementary extension denoted by $(\beta_{\A}X, \bar\A)$. This elementary extension has this property that every $\A$-measurable mapping from $X$ into any compact topological space $Y$ has a continuous extension to $\beta_{\A}X$. In particular, whenever $\A$ is a $\sigma$-algebra in $X$, we prove that $M^*(X, \A)\simeq C(\beta_{\A}X)$, where $M^*(X, \A)$ is the set of all bounded real-valued $\A$-measurable functions on $X$. We also observe that, $\beta_{\A}X=\beta X$ if and only if $\A=Z(X)$.

In section 4, we characterize real maximal ideals of $M(X, \A)$ in terms of elementary sets. We then  introduce a realcompact  extension of an elementary space $(X, \A)$, denoted by $(\upsilon_{\A}X, \bar\A_R)$. It turns out that $(\upsilon_{\A}X, \bar\A_R)$ is the smallest realcompact extension of $(X, \A)$ and $(X, \A)$ is realcompact if and only if $X=\upsilon_{\A}X$.
 Moreover, every $\A$-measurable mapping from $(X, \A)$ into any realcompact topological space $Y$ extends continuously to $\upsilon_{\A}X$. In particular, whenever $\A$ is a $\sigma$-algebra, we observe that $M(X, \A)\simeq M(\upsilon_{\A}X, \bar\A_R)$. Finally, we show that every metric space equipped with the $\sigma$-algebra of all Borel sets is a realcompact measurable space.

Section 5, is due to investigation of (the maximal) rings of quotients of $M(X, \A)$, for an elementary space $(X, \A)$. We show that, the maximal ring of quotients of $M(X, \A)$, $\mathcal Q(M(X, \A))$, is the direct limit of locally constant (continuous) functions on dense open subsets of $(X, \tau_{\A})$. In particular, whenever $\A$ is a $\sigma$-algebra, we deduce that $\mathcal Q(M(X, \A))=\mathcal Q(X)$.
 Finally, in section 6, we characterize those measurable spaces $(X, \A)$ for which $M(X, \A)\simeq C(Y)$, for some topological space $Y$.

In this article, all rings are considered to be commutative unitary rings, ideals are proper ideals and topological spaces are assumed to be completely regular Hausdorff (Tychonoff) spaces. The set complement of any subset $A$ of a set $X$, is denoted by $A^c$. The cardinal of a set $X$ is denoted by $|X|$.  A subsets of a topological space that is simultaneously clopsed and open is called {\it clopen} and the set of all clopen subsets of a space $X$ is denoted by $\text{CO}(X)$.
 We denote by $C(X)$ (or $C(X, \tau)$ to indicate the topology on $X$) the ring of all real-valued continuous functions on a topological  space $X$ and $C^*(X)$ is the subset of $C(X)$ consisting of bounded functions. For each $f \in \mathbb R^X$ the {\it zero-set} $Z(f)$ is the set of zeros of $f $ and its complement, $X\setminus Z(f)$, is called the {\it cozero-set} of $f$ and it is denoted by $\text{coz\,}f$. The set of all zero-sets in (a topological space) $X$ is denoted by $Z(X)$ and for each subset $S\subseteq \mathbb R^X$, $Z[S]$ is the set of all zero-sets of the form $Z(f)$, where $f \in S$. For any subset $A$ of a set $X$, the characteristic function of $A$ is denoted by $\chi_A$ and is defined as follows:
\begin{equation*}
\chi_A = \left\{
\begin{array}{rl}
1 & \text{if } x \in A\\
0 & \text{if } x \notin A.
\end{array} \right.
\end{equation*}
Let $\beta X$ and $\upsilon X$ be the  Stone-$\check{C}$ech compactification and the Hewitt realcompactification of $X$ respectively. For every $f \in C^*(X)$ (resp., $f \in C(X)$)  the unique extension of $f$ to a continuous function in $C(\beta X)$ (resp., $C(\upsilon X)$) is denoted by $f^\beta$ (resp., $f^\upsilon$). For each $p \in \beta X$; $M^p = \{f \in C(X) : p \in \text{cl}_{\beta X}Z(f)\}$ is a maximal ideal of $C(X)$. In fact, every maximal ideal of $C(X)$ is precisely of the form $M^p$, $p \in\beta X$.

\section{Preliminaries}
First, let us recall some definitions and notations from \cite{R}. A collection  $\A$ of subsets of a set $X$ is called an {\it algebra} (or a {\it filed} as  in \cite{W1}) in $X$, if it satisfies the following conditions:
\begin{enumerate}
\item $X\in\A$,
\item If $A\in \A$, then $A^c\in\A$,
\item If $A_1, A_2\in\A$, then $A_1\cup A_2\in\A$.
\end{enumerate}
Then the pair $(X, \A)$  is called  an elementary measurable space, or simply an  elementary space, and the members of $\A$ are called elementary  sets. Moreover, if for any collection $\{A_i\}_{i=1}^\infty\subseteq\A$ we have $\bigcup_{i=1}^\infty A_i\in\A$, then $\A$ is called a  $\sigma$-algebra, $(X, \A)$    a  measurable space and the members of $\A$ are called measurable sets.
 An intersection of any family of (resp., $\sigma$-) algebras containing a given collection $\mathcal B$ of subsets of a set $X$, is again an (resp., $\sigma$-) algebra. Thus, for any collection $\mathcal B$ of subsets of a set $X$, the  smallest ($\sigma$-) algebra  that contains $\mathcal B$ exists, see also \cite[Theorem 1.10]{R}. The smallest $\sigma$-algebra   containing a collection $\mathcal B$ of subsets of a set $X$ is denoted by $<\mathcal B>$. Following \cite[2.7]{W1}, an algebra $\A$ in a set $X$ is said to be {\it reduced}, if for every pair of distinct points of $X$ there is a member of $\A$ containing one and not the other; see also \cite[Definition 5.4]{EMY}. An elementary space $(X, \A)$ is called reduced, if its algebra $\A$ is reduced. 

Let $(X, \A)$ be a reduced elementary space. Since the algebra $\A$ is closed under finite intersections, it naturally serves as a base for a topology   on $X$. We denote by $\tau_{\A}$   the topology generated by the algebra $\A$. Henceforth, whenever an algebra $\A$ is defined in a set $X$, the set $X$ will be considered  both as  an elementary space and as a topological space.  Clearly every elementary set  $A\in\A$ is clopen in this topology and since every clopen set in a topological space, is a zero-set, thus $(X, \tau_{\A})$ is a completely regular Hausdorff   space, by \cite[Theorem 3.2]{GJ}. In fact, it is a $P$-space (i.e., a space whose every zero-set is open) whenever $\A$ is a $\sigma$-algebra; see \cite[4J]{GJ}, for more details on $P$-spaces. The following lemma also  answers a question  raised in \cite{ABS}; see Question 5.6 in the same reference.
\begin{lemma}\label{Z(X)=B(X)}
If $(X, \A)$ is a reduced measurable space, then  $(X, \tau_{\A})$ is a $P$-space.
\end{lemma}
\begin{proof}
Let $Z=Z(f)$ for some $f\in C(X)$ and $p\in Z$. Since $\A$ is a base for the open sets, for every $n\ge1$, there exists a measurable set $A_n\in\A$ such that $p\in A_n\subseteq f^{-1}(-\frac{1}{n}, \frac{1}{n})$. Hence $A_0=\bigcap_{n=1}^\infty A_n\in\A$ and $p\in A_0\subseteq Z$. This shows that $Z$ is open (clopen) and then by \cite[4J]{GJ}, we conclude that $(X, \tau_{\A})$ is a $P$-space.
\end{proof}
The assumption that $\A$ is a $\sigma$-algebra is  necessary   in the preceding lemma, as demonstrated by the following example. 

\begin{example}
Let $X=\mathbb R$ be the set of all real numbers and $\A$ be the  algebra generated by the family $\{[a, b) : a, b\in \mathbb R \text{ and } a<b \}$. Now, clearly $\tau_{\mathbb R}\subseteq \tau_{\A}$, where $\tau_{\mathbb R}$ is the usual topology on $\mathbb R$. Therefore, the identity function $i : (X, \tau_{\A})\to \mathbb R$ is continuous, but $Z(f)=\{0\}$ is not a clopen subset of $ (X, \tau_{\A})$.
\end{example}
Let  $(X, \A)$ be an elementary space and $S\subseteq X$. Then  it is easy to see that $\mathcal A|_S:=\{A\cap S : A\in \mathcal A\}$ is an  algebra  in   $S$. The subset $S$ with this  algebra is called an {\it elementary subspace} of $(X, \mathcal A)$ and is denoted by $(S, \A|_S)$. In this case, it is clear that $(S, \tau_{\A|_S})$ is a topological subspace of $(X, \tau_{\A})$, i.e., it has the subspace topology.   In addition, if $S$ is a (topologically) dense subset of $X$, then the latter  is called an {\it elementary extension} or simply an {\it extension} (when there is no ambiguity) of $(S, \A|_S)$. Note that, if $\A$ is a $\sigma$-algebra in $X$, then so is $\A|_S$, and in this case the pair $(S, \A|_S)$ is called a {\it measurable subspace} of $X$ and the latter is called a {\it measurable extension} of $S$, whenever $S$ is dense in $X$.

\begin{example}
Let $\mathfrak M$ be the $\sigma$-algebra of Lebesgue measurable sets of $\mathbb R^n$. Then it is easy to see that $\mathfrak M|_{\mathbb R^m}$ is the $\sigma$-algebra of Lebesgue measurable sets of $\mathbb R^m$, for any $1\le m<n$.
\end{example}

 The following lemma provides a characterization of the elementary sets in  an  extension $T$ of  a given elementary space $X$.
\begin{lemma}\label{m-sets of extensions}
Let $(T, \A')$ be an  extension of  an elementary space $(X, \A)$, then 
$$\A'=\{\text{cl}_T A : A\in \A \}.$$
\end{lemma}
\begin{proof}
 Since $(T, \A')$ is an  extension of $X$,  $ A'\cap X\in\A$, for any $ A'\in \A'$. On the other hand, for any $A\in \A$ there is an elementary set $ A'\in\A'$ such that $A= A'\cap X$ and thus $\text{cl}_T A=\text{cl}_T ( A'\cap X)= A'$, as $X$ is dense in $T$ and $ A'$ is clopen in $(T, \tau_{\A'})$.
\end{proof}

\begin{definition}
Let $(X, \A)$ and $(Y, \A')$ are two elementary spaces. A mapping $f: (X, \A) \to (Y, \A')$ is called  an elementary homomorphism if $f^{-1}( A')\in\A$, for any $ A'\in\A'$.
\end{definition}
An elementary isomorphism between two elementary spaces $(X, \A)$ and $(Y, \A')$, is a bijective mapping $\varphi : (X, \A) \to (Y,\A')$ such that  $\varphi$ and $\varphi^{-1}$ are both  elementary homomorphisms; see also \cite[Definition 2.3]{AABS},  for the same notion between measurable spaces. Elementary spaces that are elementary isomorphic,  considered to be identical.

\begin{definition}
Let $(X, \A)$ be an elementary space, $Y$ a topological space and $f : X\to Y$ a mapping. Then $f$ is called $\A$-measurable if $f^{-1}(V)\in\A$, for every open subset $V\subseteq Y$.
\end{definition}

For   an elemetary space $(X, \A)$, the family of all real-valued $\A$-measurable functions on $X$ is denoted by $M(X, \A)$. It is easy to see that $M(X, \A)$ is a ring (in fact a lattice-ordered $\mathbb R$-algebra) under pointwise addition and multiplication. In fact, $M(X, \A)$ is a subalgebra of $C(X, \tau_{\A})$, as every member of $\A$ is open in $(X, \tau_{\A})$. In what follows, $M(X, \A)$ stands for the ring of all real-valued $\A$-measurable functions on an elementary space $(X, \A)$.

Using \cite[Proposition 5.6]{EMY}, for any measurable space $(X, \A)$ there is a reduced measurable space $(Y, \A')$ such that $M(X, \A)\simeq M(Y, \A')$. Thus, from an algebraic viewpoint, we may always assume that the background measurable space is reduced. The same fact holds in any elementary space, as it is straightforward   to verify that the  proof of  \cite[Proposition 5.6]{EMY} and its preceding lemma carries over without modification. Just note  that, if $(X, \A)$ is an elementary space then \cite[Proposition 3.2]{EMY} is  holds in $M(X, \A)$ and we have $\A=\{Z(f) : f\in M(X, \A)\}=:Z_{\A}[X]$. In fact, if $f\in M(X, \A)$ then $Z(f)=X\setminus f^{-1}(\mathbb R\setminus\{0\})\in\A$. On the other hand, for any $A\in\A$ we have $A=Z(1-\chi_{A})$, where $\chi_A$ is the characteristic function of $A$, which is $\A$-measurable.  Thus, in the sequel,  all given elementary spaces are  assumed to be reduced. 

The following lemma is needed in the sequel as well; see also \cite[Proposition 3.3]{EMY}, for a counterpart on a measurable space.
\begin{lemma}\label{units}
Let $(X, \A)$ be an elementary space. Then an element $f\in M(X, \A)$ is unit if and only if $Z(f)=\emptyset$.
\end{lemma}
\begin{proof}
Necessity. Suppose that $f\in M(X, \A)$ is unit, then $fg=1$, for some $g\in M(X, \A)$. This implies that $Z(f)\subseteq Z(fg)=Z(1)=\emptyset$.

Sufficiency. If $Z(f)=\emptyset$, then $g=\frac{1}{f}=h\circ f$, where $h(x)=\frac{1}{x}\in C(\mathbb R)$. Now, since  a continuous function of an $\A$-measurable function is  $\A$-measurable and $fg=1$, we are done.
\end{proof}

\section{ Compact elementary  spaces}

In this section we  study the extensions of an elementary space. More specifically, we observe that if an elementary space $(X, \A)$ is given, then there exist a unique   extension $(\beta_{\A}X, \bar\A)$ of $X$ which  contains every other  elementary extension of $X$ in a natural way. Moreover, this  extension is to the elementary space $(X, \A)$ (in the category of elementary spaces and elementary homomorphisms) as $\beta Y$ to a topological space $Y$. First, we state and prove some results on  extensions of a given elementary space $(X, \A)$, then we present this maximum elementary extension.

Recall from \cite{W1}, that a {\it Boolean algebra} is a distributive lattice $L$ which is complemented (i.e., it contains distinct elements $0$ and $1$ such that, for every  $a\in L$,  $0\le a\le1$  and  there exists $a'\in L$, where $a\wedge a'=0$ and $a\vee a'=1$).   Since every algebra is naturally a Boolean algebra (see \cite[2.7]{W1}), the standard notion of a filter in a Boolean algebra can be applied directly to an elementary space. We make this explicit in the following definition.
\begin{definition}
Let $(X, \A)$ be an elementary space. A nonempty family  $\mathcal F\subseteq \A$  is called an $\A$-filter on $X$ if it satisfies the following conditions:
\begin{enumerate}
\item $\emptyset\notin \mathcal F$,
\item if $A, B\in\mathcal F$, then $A\cap B\in\mathcal F$,
\item if $A\in\A$ and $B\in\mathcal F$ where $B\subseteq A$, then $A\in\mathcal F$.
\end{enumerate}
\end{definition}
 An $\A$-filter is said to be an $\A$-ultrafilter if it is maximal in the set of all $\A$-filters on $X$ under the set inclusion. Using this notion and since $\A=Z_{\A}[X]$, by the argument preceding Lemma \ref{units}, one obtains a natural   duality between  $\A$-filters on an elementary space $(X, \A)$  and  ideals in the ring $M(X, \A)$; see also \cite[Theorem 3.5]{EMY}. Consequently, any result   stated in terms of $\A$-filters can be readily translated into the language  of ideals. Now, using  \cite[Proposition 2.6]{W1} and Zorn's lemma, we have the following lemma which is needed in the sequel. 
\begin{lemma}\label{A-ultrafilter}
Let $(X, \A)$ be an elementary space and $\mathcal F$ be an $\A$-filter on $X$. Then the following statements hold.
\begin{enumerate}
\item $\mathcal F$ is contained in an $\A$-ultrafilter.
\item $\mathcal F$ is an $\A$-ultrafilter if and only if  for each $A\in \A$, either $A\in\mathcal F$ or $A^c\in \mathcal F$.
\item If $\mathcal F$ and $\mathcal G$  are distinct $\A$-ultrafilters, then there are elementary sets 
$A_1\in \mathcal F$ and $A_2\in \mathcal G$ such that $A_1\cap A_2=\emptyset$. 
\end{enumerate}
\end{lemma}
A central problem in the theory of extensions of a topological space $X$ is that of extending the mappings belonging to a certain class of mappings on $X$. The basic tool for extending  continuous functions on a topological space $X$ is the concept of  $z$-filter. Specifically, this  happened by adjoining to $X$ one new point for each free $z$-ultrafilter, to serve as its limit. The result is the Stone-$\check C$ech compactification $\beta X$ of $X$, in which every continuous mapping from $X$ into any compact topological space $Y$ has a continuous extension. This leads to the following definition and terminologies.

Let $(X, \A)$  be an elementary subspace of  an elementary space   $(T, \A')$.
\begin{itemize}
\item  An $\A$-filter $\mathcal F$ (resp., ideal $I$ in $M(X, \A)$) on $X$ is called fixed if $\bigcap_{A\in\mathcal F}A\ne\emptyset$ (resp., $\bigcap_{f\in I}Z(f)\ne\emptyset$); otherwise is called free. 
\item We say that an $\A$-filter $\mathcal F$ on $X$ is convergent to $p\in T$ if every neighborhood of $p$ (in $T$) contains a member of $\mathcal F$.
\end{itemize}
To develop a  theory for extending mappings of a certain class  on an elementary space, we first need some preliminary lemmas and  corollaries.

\begin{lemma}\label{A-ultrafilters on X}
Let $(X, \A)$  be an elementary subspace of   an elementary space  $(T, \A')$,  $A_0\in\A$ and $p\in\text{cl}_TA_0$. Then there is at least one $\mathcal A$-ultrafilter on $X$ that contains $A_0$ and convergence to $p$.
\end{lemma}
\begin{proof}
Let $\mathcal F:=\{ A'\cap X :  A'\in\A' \text{ and } p\in  A'\}$. Since $p\in\text{cl}_TA_0$,  Lemma \ref{m-sets of extensions} implies that $A_0\in\mathcal F$. On the other hand,  $p\in\text{cl}_TX$ follows   that $\mathcal F$ is an $\A$-filter on $X$ and so it's enough to show that it is an $\A$-ultrafilter, as clearly $\mathcal F$ is convergent to $p$. To let this end, suppose that $A\in\A$, thus there exists an elementary set $ A'\in \A'$ such that $A= A'\cap X$, as  $X$ is an elementary subspace of   $T$. On the other hand,
$$p\in\text{cl}_TX=\text{cl}_T(A\cup (X\setminus A))=\text{cl}_TA\cup\text{cl}_T(X\setminus A),$$
  implies that $p\in \text{cl}_TA$ or $p\in\text{cl}_T(X\setminus A)$. If $p\in\text{cl}_TA$,  Then
$$p\in \text{cl}_TA=\text{cl}_T( A'\cap X)\subseteq A',$$
which it follows  that $A\in\mathcal F$. The other case implies that $X\setminus A\in\mathcal F$. This  shows that $\mathcal F$ is an $\A$-ultrafilter by Lemma \ref{A-ultrafilter}(2).
\end{proof}
\begin{corollary}\label{Extensions}
If  $(T, \A')$ is an  extension of an elementary space $(X, \A)$, then every point of $T$ is the limit of a unique  $\mathcal A$-ultrafilter 
$$\A^p=\{ A'\cap X :  A'\in\A' \text{ and } p\in  A' \}=\{A\in \A : p\in\text{cl}_TA\}, $$
 on $X$. Moreover, if $\A'$  is a $\sigma$-algebra, then  $\A^p$ is closed under countable intersection.
\end{corollary}
\begin{proof}
 Since $X$ is  dense in $T$,  every point of $T$ is the limit of at least one $\mathcal A$-ultrafilter on $X$, by Lemma \ref{A-ultrafilters on X}. On the other hand, using Lemma \ref{A-ultrafilter}(3), distinct $\A$-ultrafilters on $X$ contains disjoint members of $\A$. Thus, distinct $\A$-ultrafilters on $X$ can not have a common limit point in $T$.  The last statement is evident.
\end{proof}

\begin{definition}\label{compact m-space}
An elementary space $(X, \A)$ is called compact if every $\A$-ultrafilter on $X$ is fixed.
\end{definition}
For an elementary space $(X, \A)$, since the algebra $\A$ is a base for the open (or closed) sets of the associated topological space, a straightforward argument yields the following lemma. Remember that a topological space $X$ is compact if and only if every filter consisting of closed sets is fixed, see \cite[Theorem 3.1.1]{G}.
\begin{lemma}\label{T-compactness}
If $(X, \A)$ is an elementary space, then the following are equivalent.
\begin{enumerate}
\item  The elementary space $(X, \A)$ is compact.
\item The topological  space $(X, \tau_{\A})$ is compact.
\item  Every $\A$-filter on $X$ is fixed.
\end{enumerate}
In this case, since $(X, \A)$ is reduced, then the intersection of members of any $\A$-ultrafilter is a singleton.
\end{lemma}

The following  theorem discusses  extending of elementary homomorphisms.
\begin{theorem}\label{m-homo extension}
Let $(T, \A')$ be an  extension of an elementary space $(X, \A)$. Then every elementary homomorphism $\varphi $ from $X$ into any reduced compact elementary space  $(Y, \tilde\A)$  extends to an elementary homomorphism  $\bar \varphi $ on $T$.
\end{theorem}
\begin{proof}
First, to define $\bar\varphi$ on $T$, let $p\in T$ and $\A^p$ be the unique $\A$-ultrafilter on $X$ converges to $p$, which obtained by Corollary \ref{Extensions}. Define 
$$\varphi^{\#}\A^p=\{\tilde A\in \tilde\A : \varphi^{-1}(\tilde A)\in\A^p \}.$$
It is clear that $\varphi^{\#}\A^p$ is an $\A$-ultrafilter (on $Y$)  by Lemma \ref{A-ultrafilter}(2). Now, using Lemma \ref{T-compactness}, $\bigcap \varphi^{\#}\A^p$ is a singleton. Let $\bar\varphi(p)$ be this single point, i.e., $\bigcap \varphi^{\#}\A^p=\{\bar\varphi(p)\}$.

Next, to see that $\bar\varphi$ is an extension of $\varphi$, note that for any $x\in X$,
$$\varphi^{\#}\A_x=\{\tilde A\in \tilde\A : \varphi^{-1}(\tilde A)\in\A_x\}=\{\tilde A\in \tilde\A : \varphi(x)\in \tilde A\}=:\tilde\A_{\varphi(x)}.$$
Clearly, $\bigcap\tilde\A_{\varphi(x)}=\{\varphi(x)\}$, which means $\bar\varphi(x)=\varphi(x)$.

Finally, we show that $\bar\varphi$ is an elementary homomorphism. Using the definition of $\bar\varphi$, for any $\tilde A\in \tilde\A$ we have 
$$\bar\varphi^{-1}(\tilde A)=\{ p\in T : \bar\varphi(p)\in \tilde A \}=\{ p\in T : \varphi^{-1}(\tilde A)\in \A^p \}=\text{cl}_T\varphi^{-1}(\tilde A),$$
 by  Corollary \ref{Extensions}. Also, the last term belongs to $ \A'$, by Lemma \ref{m-sets of extensions},  which means that $\bar\varphi$ is an elementary homomorphism.
\end{proof}
The following  theorem discuss  extending of $\A$-measurable mappings on an elementary space.
\begin{theorem}\label{Extension into compact}
Let  $(T, \A')$  be an  extension of an elementary space $(X, \A)$. Then every $\A$-measurable mapping $\varphi$ from $X$ into any  compact topological space $Y$, extends to a continuous mapping $\bar \varphi$ on $(T, \tau_{\A'})$. Moreover, if $\A'$ is a $\sigma$-algebra and $Y$ is second countable, then the extended mapping $\bar\varphi$  is $\A'$-measurable.
\end{theorem}
\begin{proof}
First, to define $\bar\varphi$ on $T$, let $p\in T$ and $\A^p$ be the unique $\A$-ultrafilter on $X$ converges to $p$, which obtained by Corollary \ref{Extensions}. Define
$$\varphi^{\#}\A^p=\{ Z\in Z(Y) : \varphi^{-1}(Z)\in\A^p \}.$$
Note that, since $\varphi$ is $\A$-measurable, $\varphi^{-1}(Z)\in\A$, for any $Z\in Z(Y)$.
It is easy to see that $\varphi^{\#}\A^p$  is a prime $z$-filter on $Y$. It follows then, $\varphi^{\#}\A^p$ has a limit point in $Y$, by \cite[Theorems 3.17 and 4.11]{GJ}. Take this limit point as $\bar\varphi(p)$, i.e., 
$$\bigcap \varphi^{\#}\A^p=\{\bar\varphi(p)\}.$$
This defines a mapping $\bar \varphi$ from $T$ into $Y$. Whenever $p\in X$ then, $Z\in \varphi^{\#}\A^p$  if and only if $p\in \varphi^{-1}(Z)$ or equivalently $\varphi(p)\in Z$ and this implies that $\varphi(p)\in\bigcap \varphi^{\#}\A^p$, i.e., $\bar\varphi(p)=\varphi(p)$. This shows that $\bar \varphi$ is an extension of $\varphi$.

 Next, to show that $\bar\varphi$ is continuous at $p\in T$, suppose that $ U$ is a neighborhood of $\bar\varphi(p)$ in $Y$. Since $Y$ is completely regular, there are  zero-sets $\tilde Z_0, \tilde Z_1\in Z(Y)$  such that $\bar\varphi(p)\in Y\setminus \tilde Z_1 \subseteq\tilde Z_0\subseteq U$, see \cite[Theorem 3.2 and 3.2(b)]{GJ}. This  implies that $\tilde Z_1\notin \varphi^{\#}\A^p$, and since $\varphi^{\#}\A^p$ is prime and $\tilde Z_0 \cup\tilde Z_1=Y$, we conclude that $\tilde Z_0\in\varphi^{\#}\A^p$. Therefore $\varphi^{-1}(\tilde Z_0)\in \A^p$ and since $X$ is an elementary subspace of $(T, \A')$, we may choose an elementary set  $A'_0\in \A'$ such that $\varphi^{-1}(\tilde Z_0)=A'_0\cap X$. To complete the proof, it is enough to show that $\bar\varphi(A'_0)\subseteq U$,  as $A'_0$ is a neighborhood of $p$. Let $q\in A'_0$. Thus $A'_0\cap X\in\A^q$, by Corollary \ref{Extensions}, and hence $\tilde Z_0 \in\varphi^{\#}\A^q$. It follows then $\bar\varphi(q)\in \tilde Z_0\subseteq U$ and we are done.

Finally,  for the last part of the theorem,  suppose that  $\A'$ is a $\sigma$-algebra and  $Y$ is second countable, then $Y$  is metrizable by \cite[Theorem 4.2.8]{E}. Thus, every open subset $U\subseteq Y$ is a countable union of open balls, as $Y$ is second countable; see also \cite[Theorem 1.1.15]{E}. Using this, it is enough to show that $\bar\varphi^{-1}(B(y, r))\in\A'$, for any  open ball $B(y, r)$. If $U=B(y_0, r_0)$ is an open ball  with center at some point $y_0\in Y$ and radius $r_0$, we may write
$$U=\bigcup_{n=1}^{\infty}B(y_0, r_n)=\bigcup_{n=1}^{\infty}\overline{B(y_0, r_n)},$$
where $r_n$ is an ascending sequence of positive real numbers converges to $r_0$ and in here  $\overline{B(y_0, r_n)}$ is the closure of $B(y_0, r_n)$ in $Y$. Now, since $\varphi$ is  $\A$-measurable and $(X, \A)$ is an elementary subspace of $(T, \A')$, there is a measurable set $A'_n\in\A'$ such that $\varphi^{-1}(B(y_0, r_n))=A'_n\cap X$, for every $n\ge1$. Using the continuty of $\bar\varphi$, it follows that
$$\bar\varphi(A'_n)=\bar\varphi(\text{cl}_T(A'_n\cap X))\subseteq\text{cl}_Y\varphi(A'_n\cap X)\subseteq\overline{B(y_0, r_n)}.$$
 Therefore $A'_0=\bigcup_{n=1}^{\infty}A'_n\subseteq\bar\varphi^{-1}(U)$ and $A'_0\in\A'$. For the reverse inclusion, note that 
$$\varphi^{-1}(U)=\varphi^{-1}(\bigcup_{n=1}^{\infty}{B(y_0, r_n)})=\bigcup_{n=1}^{\infty}A'_n\cap X=A'_0\cap X,$$
 which implies that $\varphi(X\setminus A'_0)\subseteq Y\setminus U$. Since $\text{cl}_T(X\setminus A'_0)=T\setminus A'_0$, so  by  the  continuty of $\bar\varphi$ we conclude that $\bar\varphi(T\setminus A'_0)\subseteq Y\setminus U$ which means $\bar\varphi^{-1}(U)\subseteq A'_0$, and the proof is complete.
\end{proof}
Now we introduce the maximum elementary extension of a given (reduced) elementary space $(X, \A)$.
 Suppose that $S(\A)$ is the family of all $\A$-ultrafilters (on $X$) and  $\beta_{\A}X$ is an index set for this family, where for each point $p\in \beta_{\A}X$ there corresponds exactly one $\A$-ultrafilter, denoted by $\A^p$, in $S(\A)$ and vice-versa. Then,  by \cite[Proposition 2.9]{W1}, the family $\bar\A:=\{\bar A : A \in \A \}$, where $\bar A=\{p\in \beta_{\A} X : A\in\mathcal A^p \}$, constitutes a reduced algebra isomorphic to $\A$ (as a Boolean algebra) in the set $\beta_{\A}X$.  
Using the proof of the Stone Representation Theorem (\cite[Theorem 2.10]{W1}), the topology generated by $\bar\A$ turns $\beta_{\A}X$ into a compact totally disconnected (zero-dimensional) Hausdorff space. This reduced compact  elementary space $(\beta_{\A}X, \bar\A)$, contains $X$ in a natural way. In fact, for each point $x\in X$,  the subfamily $\A_x=\{ A\in \A : x\in A\}$ is an (fixed) $\A$-ultrafilter on $X$, by Lemma \ref{A-ultrafilter}. Now, if we denote the set of all fixed $\A$-ultrafilters on $X$ by $\mathcal X$, it is easy to see that the mapping $x\to \A_x$ is an elementary isomorphism between $(X, \A)$ and $(\mathcal X, \bar\A|_{\mathcal X})$. Thus we may consider $X$ as an elementary subspace of $\beta_{\A}X$ and clearly $\bar A\cap X=A$, for any $A\in\A$. Using   this,  every $\A$-ultrafilter $\A^p$ on $X$, converges to $p\in\beta_{\A}X$ and in particular      $\bar A=\text{cl}_{\beta_{\A}X}A$, for any $A\in\A$. This  shows  that, $(\beta_{\A}X, \bar\A)$ is a compact   elementary extension of $(X, \A)$; see also \cite[\S2]{ABS}.
\begin{corollary}\label{beta X}
Every (reduced) elementary space $(X, \A)$ has  a reduced compact    extension $(\beta_{\A}X, \bar\A)$, which has the following properties.
\begin{enumerate}
\item Every elementary homomorphism from $X$ into any reduced compact elementary space $Y$ extends to an elementary homomorphism  on $\beta_{\A}X$.
 \item Every $\A$-measurable mapping $\varphi$ from $X$ into any compact topological space $Y$, has a continuous extension $\bar\varphi$ on $\beta_{\A}X$.
\item Every  extension $(T, \A')$ of $X$, is an elementary subspace of $\beta_{\A}X$, i.e., $X\subseteq T\subseteq \beta_{\A}X$, which means $(\beta_{\A}X, \bar\A)$ is the maximum  extension of $X$.
\item If $(T, \A')$ is an elementary extension of $X$, then $\beta_{\A'}T=\beta_{\A}X$.
\item For any two elementary sets $A, B\in \A$, we have
$$\text{cl}_{\beta_{\A}X}(A\cap B)=\text{cl}_{\beta_{\A}X} A\cap \text{cl}_{\beta_{\A}X} B.$$
\end{enumerate}
\end{corollary}
\begin{proof}
\begin{enumerate}
\item This is due to Theorem \ref{m-homo extension}.
\item This is a consequence of Theorem \ref{Extension into compact}.
\item We define an elementary homomorphism $\varphi : (T, \A')\to (\beta_{\A}X, \bar\A)$ which is an isomorphism between $(T, \A')$  and $(\varphi(T), \bar\A|_{\varphi(T)})$. Let $p\in T$ and 
$$\mathcal F^p=\{ A'\cap X :  A'\in\A' \text{ and } p\in  A' \}, $$ 
be the unique $\A$-ultrafilter obtained by Corollary \ref{Extensions}. Using  the argument before this result, $\mathcal F^p$ has  a unique convergent point in $\beta_{\A}X$ and we let   $\varphi(p)$ to be this unique point.
 Using Corollary \ref{Extensions}, the mapping $\varphi$ just defined, is a monomorphism from $T$ into $\beta_{\A}X$. Now, if $A'\in\A'$, then $A'\cap X=A\in \A$ and  by  the argument preceding the corollary, $\bar{A}=\text{cl}_{\beta_{\A}X}A=\{q\in\beta_{\A}X : A\in\A^q\}$. Thus, another application of Corollary \ref{Extensions}, implies that $p\in A'$ if and only if $A\in\mathcal F^p$ and this is equivalent to saying that $\varphi(p)\in \text{cl}_{\beta_{\A}X}A$, i.e.
$\varphi(A')= \bar{A}\cap \varphi(T)$. Hence $T$ and $\varphi(T)$ are elementary isomorphic and hence
$(\beta_{\A}X, \bar\A)$ is an  extension of $T$.
\item Since $\beta_{\A'}T$ is an  elementary extension of $T$ (and hence  an  extension of $X$), and $\beta_{\A}X$ is an extension of $T$, the result follows from (3). 
\item Using the argument preceding the corollary, $p\in \text{cl}_{\beta_{\A}X} A\cap \text{cl}_{\beta_{\A}X} B$ implies that $A , B\in\A^p$. Now, since $\A^p$ is a $\A$-filter, we conclude that $A\cap B\in \A^p$, i.e., $p\in\text{cl}_{\beta_{\A}X}(A\cap B)$.
\end{enumerate}\end{proof}
For a given (reduced) elementary space $(X, \A)$, the  space $( \beta_{\A}X, \bar \A)$  is called the {\it Stone extension} of $(X, \A)$, which  is the only compact elementary extension of $X$. Also, whenever  $\varphi$ is a mapping from $X$ into a compact topological space $Y$, its extension $\bar\varphi$ is called  the {\it Stone extension} of $\varphi$.

Using the duality between $\A$-filters on an elementary space $(X, \A)$ and   ideals of $M(X, \A)$ and the argument before Corollary \ref{beta X}, we deduce the following characterization of maximal ideals of $M(X, \A)$, which is the counterpart of Gelfand-Kolomogroff theorem in this framework.
\begin{corollary}\label{G-K theorem}
For a given elementary space $(X, \A)$, maximal ideals of $M(X, \A)$ are of the form
$$M^p=\{f\in M(X, \A) : p\in\text{cl}_{\beta_{\A}X}Z(f) \},$$
for an unique point $p\in\beta_{\A}X$.
\end{corollary}

 Using Corollary \ref{beta X},  every $f\in M^*(X, \A)$ has a (unique) continuous extension, that we shall denote  by, $f^{\beta} \in C(\beta_{\A}X)$. Since $X$ is dense in $(\beta_{\A}X, \tau_{\bar\A})$, the mapping $f\to f^{\beta}$ is an isomorphism of $M^*(X, \A)$ into $C(\beta_{\A}X)$. In fact, by the following result, this isomorphism is onto, if $\A$ is a $\sigma$-algebra in $X$. First note that,  $M^*(X, \A)$ is a Banach algebra, by \cite[Proposition 3.1]{YEMS}, under the   uniform norm topology, i.e., $\|f\|=\text{sup}\{ |f(x)| : x\in X\}$, for every $f\in M^*(X, \A)$.

\begin{proposition}\label{M^*=C(beta)}
For any measurable space $(X, \A)$, the mapping $f\to f^{\beta}$ is a homeomorphism between $M^*(X, \A)$ and $C(\beta_{\A}X)$, as Banach algebras.
\end{proposition}
\begin{proof}
Let $\varphi :M^*(X, \A)\to C(\beta_{\A}X)$ be the mapping that takes $f$ to $f^{\beta}$. By the argument preceding the proposition, $\varphi$ is one-to-one. Suppose that $\varphi(M^*(X, \A))=R\subseteq C(\beta_{\A}X)$. Using the Stone-Weierstrass  theorem, we show that $R= C(\beta_{\A}X)$. First, note that, the algebra $R$ contains the constant functions. Next,  using  Corollary \ref{beta X}, for any two distinct points $p, q\in\beta_{\A}X$ there is a measurable set $A\in\A$ such that $\text{cl}_{\beta_{\A}X}A$ contains $p$ and not $q$. Therefore $\chi_{A}^{\beta}\in R$, $\chi_{A}^{\beta}(p)=1$ and $\chi_{A}^{\beta}(q)=0$. This shows that the algebra $R$ separates points in $\beta_{\A}X$. 

Finally,  since  any continuous function takes closter points to closter points and $X$ is dense in $\beta_{\A}X$, for every $f\in M^*(X, \A)$, we have $\|f\|=\|f^\beta\|$.  Therefore  $\{f_n\}\subseteq M^*(X, \A)$ is a cauchy sequence in the uniform norm topology on $M^*(X, \A)$ if and only if $\{f^\beta\}\subseteq R$ is so. Thus \cite[Proposition 3.1]{YEMS} implies that $R$ is uniformly closed and we are done by \cite[Theorem 44.7]{W}.
\end{proof}
The following example shows that the preceding result may fail if $\A$ is merely an algebra rather than a $\sigma$-algebra. 
\begin{example}\label{beta_A N}
Let $X=\mathbb N$ be the set of all positive integers and 
$$\A=\{ A\subseteq X : A \text{ or } A^c \text{ is finite }\}.$$
It is clear that $\A$ is an algebra in $X$ and $\A^{\omega_0}=\{A\in \A : A^c \text{ is finite }\}$ is an $\A$-ultrafilter by Lemma \ref{A-ultrafilter}(2). In fact, this is the only free $\A$-ultrafilter on $X$. For, if $\mathcal F$ is a distinct $\A$-ultrafilter on $X$, then using Lemma \ref{A-ultrafilter}(3),  there is an elementary set $A\in \mathcal F\setminus\A^{\omega_0}$, i.e., $A$ is finite. Now, since $\mathcal F$ is closed under finite intersection, it is fixed. Using this, we have $\beta_{\A}X=\mathbb N^*=\mathbb N\cup\{\omega_0\}$, the one-point compactification of $\mathbb N$ and 
$$\bar\A=\{A\subseteq\mathbb N : A \text{ is finite}\}\cup\{\{\omega_0\}\cup A : A^c \text{ is finite}\}.$$

Now, it is easy to see that $M^*(X,\A)$ consists only of  real-valued sequences that are finally constant, whereas $C( \mathbb N^*)$ consists of all real-valued convergent sequences.
\end{example}
In the following result, we compare $\beta_{\A}X$ with $\beta X$ and determine precisely when $\beta_{\A}X=\beta X$.
\begin{proposition}\label{beta_AX=beta X}
Let $(X, \A)$ be an elementary space. Then the following statements are equivalent.
\begin{enumerate}
\item $\beta_{\A}X=\beta X$.
\item Any two disjoint zero-sets of $X$ separated by two disjoint elementary sets.
\item $M^*(X, \A)=C^*(X, \tau_{\A})$.
\item $M(X, \A)=C(X, \tau_{\A})$.
\item $\A= Z(X)$.
\end{enumerate}
\end{proposition}
\begin{proof}
(1) $\Rightarrow$ (2). Suppose that $Z_1$ and $ Z_2$ are two disjoint zero-sets of $X$. Thus, by the assumption and \cite[Theorem 6.5(IV)]{GJ}, we have $\text{cl}_{\beta_{\A}X}Z_1\cap \text{cl}_{\beta_{\A}X}Z_2=\emptyset$. Now, since $\bar\A$ is a base for open (closed) sets of $\beta_{\A}X$, for every $p\in \text{cl}_{\beta_{\A}X}Z_1$, there is an elementary set $A_p\in \A$ such that $p\in \text{cl}_{\beta_{\A}X}A$ and $\text{cl}_{\beta_{\A}X}A\cap \text{cl}_{\beta_{\A}X}Z_2=\emptyset$. Using the compactness of $\text{cl}_{\beta_{\A}X}Z_1$ there is a finite set of  points $p_1, ..., p_n\in\text{cl}_{\beta_{\A}X}Z_1$ such that 
$$\text{cl}_{\beta_{\A}X}Z_1\subseteq \bigcup_{i=1}^n\text{cl}_{\beta_{\A}X}A_{p_i}\subseteq \beta_{\A}X\setminus \text{cl}_{\beta_{\A}X}Z_2.$$
Thus $A_0=\bigcup_{i=1}^nA_{p_i}\in \A$, $Z_1\subseteq A_0$ and $Z_2\subseteq A_0^c$.

(2) $\Rightarrow$ (3). Suppose that $Z_1$ and $ Z_2$ are two disjoint zero-sets of $X$. Thus, by the hypothesis, we may assume that there is $A_0\in \A$ such that $Z_1\subseteq A_0$ and $Z_2\subseteq A_0^c$. It follows then, $\chi_{A_0}(Z_1)=1$ and $\chi_{A_0}(Z_2)=0$. Now, since $M^*(X, \A)$ is closed in the uniform topology by \cite[Proposition 3.1]{YEMS}, the result follows from \cite[44A.1]{W}.

(3) $\Rightarrow$ (4). Using (3), we have $\A=Z(X)$. Now, if $f\in C(X)$ and $U\subseteq \mathbb R$ is an open set, then $f^{-1}(U)$ is a cozero-set of $X$. Thus $f^{-1}(U)\in \{X\setminus Z : Z\in Z(X)\}=Z(X)=\A$, which means that $f$ is $\A$-measurable, i.e., $M(X, \A)=C(X, \tau_{\A})$.

(4) $\Rightarrow$ (5). It is evident.

(5) $\Rightarrow$ (1). Using the hypothesis and the argument preceding Corollary \ref{beta X}, any two disjoint zero-sets of $X$ have disjoint closures in $\beta_{\A}X$. Now, since $\beta_{\A}X$ is a compactification of $X$,  the result follows from    \cite[Theorem 6.5(II)]{GJ}.
\end{proof}

We close this section with a brief discussion on what is so called {\it compact measurable} space in \cite{EMY}. First, remember that an element $a$ in a lattice $L$ is called {\it compact} provided that if $a=\bigvee S$, for some subset $S\subseteq L$, then there is a finite subset $T$ of $S$ such that $a=\bigvee T$. In  \cite{EMY}, the authors  called a  measurable space $(X, \A)$, compact if $X$ is a compact element in $\A$, as a lattice. It is worth noting that if $\A$ is a $\sigma$-algebra, then  $(X, \A)$ is a compact elementary space in our sense if and only if it is a compact measurable space in the sense of \cite{EMY}, see  Proposition 4.11 in the same reference.

The  definition of a compact measurable space  (in  the sense of \cite{EMY})  may be generalized in a natural way  to elementary spaces. Nevertheless, as we see in the following, under this generalization, certain results in the literature fail to carry over, while others continue to hold with a slight modification. An elementary space $(X, \A)$ is called {\it $\mathcal L$-compact} if $X$ is a compact element of  the lattice $\A$. Clearly, every $\mathcal L$-compact elementary space $(X, \A)$, is compact. The convers of this fact is discussed in the following result.
\begin{lemma}\label{T-compact is L-compact}
A compact elementary space $(X, \A)$ is $\mathcal L$-compact if and only if, the only dense open subset of $(X, \tau_{\A})$ is  $X$ itself.
\end{lemma}
\begin{proof}
Just note that $\bigvee_{i\in I}A_i=X$, for a subfamily $\{A_i\}_{i\in I}$ of $\A$, if and only if $\bigcup_{i\in I}A_i$ is a dense open subset of $(X, \tau_{\A})$.
\end{proof}
One may think that for a given elementary space $(X, \A)$, the Stone extension $(\beta_{\A}X, \bar\A)$ is  $\mathcal L$-compact. However, this is not the case; for instance, if $(X, \A)$ is the elementary space discussed in Example \ref{beta_A N}, then $\bigvee_{n\in\mathbb N}\{n\}=\mathbb N^*=\beta_{\A}X$, but there is no finite subfamily $T$ of $\{\{n\}\}_{n\in\mathbb N}$ such that $\bigvee T=\beta_{\A}X$.  More generally, the following proposition, shows that in the absence of measurable cardinals, for no infinite elementary space $(X, \A)$, $\beta_{\A}X$ is $\mathcal L$-compact; see Chapter 12 n \cite{GJ}, for more details on nonmeasurable cardinals.
\begin{proposition}
Let  $(X, \A)$ be an elementary space and suppose that $|X|$ is nonmeasurable. Then $(X, \A)$ is  $\mathcal L$-compact if and only if $X$ is finite and $\A=\mathcal P(X)$.
\end{proposition}
\begin{proof}
If $(X, \A)$ is  $\mathcal L$-compact, then it is compact as well. 
Now, using Lemma \ref{T-compact is L-compact}, $X$ is $\mathcal L$-compact if and only if $X$  is the only dense open subset of itself. But this implies that $\mathcal Q(X)=C(X)$, where $\mathcal Q(X)$ is the maximal ring of quotients of $C(X)$; see \cite[Theorem 2.6]{FGL}. Now, by the argument in \cite[3.5]{FGL},  it follows from the last equality that $(X, \tau_{\A})$ is a discrete space. Hence by the compactness, we conclude that $X$ is finite  and $\A=\mathcal P(X)$. The convers is evident.
\end{proof}
The preceding proposition,  is a generalization of \cite[Proposition 3.2]{EM1} to elementary spaces, in the absence of measurable cardinals. However, Proposition 2.3 in the same reference (see also \cite[Proposition 4.11]{EMY}) may fail  in this  framework; for instance if $(X, \A)$ is an infinite (with nonmeasurable cardinal) elementary space, then every proper ideal of $M(\beta_{\A}X, \bar\A)$ is  fixed, but $(\beta_{\A}X, \bar\A)$ is not  $\mathcal L$-compact, by the preceding result. In addition, if $(X, \A)$ is the elementary space discussed in Example \ref{beta_A N}, then $M(X, \A)=M^*(X,\A)$ but $(X, \A)$ is not even compact. This shows that \cite[Proposition 3.7]{EM1}  may also fail  in the context of elementary spaces.

\section{Realcompact elementary spaces}    

In this section,  we generalize and develop the theory of realcompact measurable spaces, introduced in \cite{AABS}, to the setting of  elementary spaces. Specifically, we  characterize the real maximal ideals of $M(X, \A)$ in terms of elementary sets. We then introduce a realcompact extension  of an elementary space $(X, \A)$, which serves as   a fundamental tool in the theory of realcompact elementary spaces. Moreover, it turns out that, this extension is the largest subspace of $\beta_{\A}X$ to which every $f\in M(X, \A)$ admits a real-valued continuous extension. Finally, we show that every metric space equipped with the $\sigma$-algebra of Borel sets is a realcompact measurable space.  First, we need the following lemma which is a generalization of \cite[Theorem 2.3]{ABS} to elementary functions.
\begin{lemma}\label{multiple elements}
Let $(X, \A)$ be an elementary space and $f, g\in M(X, \A)$. Then $Z(g)\subseteq Z(f)$ if and only if $f$ is a multiple of $g$.
\end{lemma}
\begin{proof}
We only need  to prove the necessity. Thus, let $Z(g)\subseteq Z(f)$ and define 
\begin{equation*}
h(x) = \left\{
\begin{array}{rl}
0 & \text{if } x \in Z(g)\\
\frac{1}{g(x)} & \text{if } x\notin Z(g)
\end{array} \right.
\end{equation*}
Now, it is enough to show that $h\in M(X, \A)$, as $f=ghf$. To prove this, let $U$ be an open subset of $\mathbb R$ and $t(x)=\frac{1}{x}\in C(\mathbb R\setminus\{0\})$. Then, $t^{-1}(U)$ is open in $\mathbb R$. It follows then $h^{-1}(U)= g^{-1}(t^{-1}(U))\cap \text{coz\,} g\in\A$, whenever $0\notin U$ and   $h^{-1}(U)= Z(g)\cup (g^{-1}(t^{-1}(U\setminus\{0\}))\cap \text{coz\,} g)\in\A$, if $0\in U$.
\end{proof}

Using the preceding lemma, we conclude that every ideal of $M(X, \A)$ is absolutely convex. Remember that an ideal $I$ in a lattice-ordered ring $R$ is absolutely convex provided that $a\in I$, whenever $|b|\le |a|$ for some $ b\in I$.
 Also, it is a straightforward to verify that the proof of Theorems 3.4 and 3.5 in \cite{ABS}, applies verbatim in this new setting and an ideal $I$ of $M(X, \A)$ is maximal if and only if $M(X, \A)/I$ is totally ordered. Since, for any ideal $I$ of $M(X, \A)$, the lattice-ordered ring $M(X, \A)/I$ contains a copy of $\mathbb R$, this leads to the following definition.
\begin{definition}\label{real ideal}
Let $(X, \A)$ be an elementary space. A maximal ideal $M$ of $M(X, \A)$ is called real provided that $M(X, \A)/M\simeq\mathbb R$, otherwise  is called hyper-real. Also, an $\A$-ultrafilter $\A$ is called real whenever its corresponding maximal ideal in $M(X, \A)$ is real.
\end{definition}
Recall that a totally ordered filed $F$ is said to be archimedean if for every element $a\in F$, there exists a positive integer $n \in \mathbb N$ such that $n \ge a$.  On the other hand, an element $a$ in a totally ordered field $F$ is called infinitely large (resp., infinitely small), if $a>n$ (resp., $0<a<\frac{1}{n}$), for every $n\in\mathbb N$. Hence, a maximal ideal $M$ of $M(X, \A)$ is real if and only if $M(X, \A)/M$ is archimedean. For an elementary space $(X, \A)$, we denote by $\upsilon_{\A}X$ the set of all points $p\in\beta _{\A}X$ such that $M^p$ (see Corollary \ref{G-K theorem}) is a real maximal ideal of $M(X, \A)$.  Clearly $X\subseteq\upsilon_{\A}X$ and thus $(\upsilon_{\A}X, \bar\A_{R})$ is an extension of $X$, where $\bar\A_{R}=\bar\A|_{\upsilon_{\A}X}$ and $\bar\A$ is the algebra in $\beta_{\A}X$. Henceforth, $\upsilon_{\A}X$ will be considered  with this algebra. Note that, by \cite[Theorem 2.4]{AABS}, $ \bar\A_{R}$ is a $\sigma$-algebra in $\upsilon_{\A}X$, whenever $(X, \A)$ is a measurable space. 
\begin{definition}\label{m-realcompact}
An elementary space $(X, \A)$ is called realcompact  if every real maximal ideal of $M(X, \A)$ is fixed.
\end{definition}
Thus, by the preceding notation, it is clear that an elementary space $(X, \A)$ is realcompact if and only if $X =\upsilon_{\A}X$. 
In \cite[Example 3.10]{ABS}, the authors show that the set of real numbers $\mathbb R$ regarded with the $\sigma$-algebra of all Lebesgue measurable sets is a realcompact measurable space. In the following, we generalize this example and show that the measurable space $(\mathbb R^n, \mathfrak M)$  is  realcompact. 
\begin{example}
Let $(\mathbb R^n, \mathfrak M)$, for some $n\in\mathbb N$, be the Euclidean space  $\mathbb R^n$ regarded with the $\sigma$-algebra of all Lebesgue measurable sets. Moreover, let $M$ be a real maximal ideal in $M(\mathbb R^n, \mathfrak M)$. Fix some index $1\le i\le n$, and define $f: \mathbb R^n\to \mathbb R$ by $f(x_1,..., x_n)=x_i$. It is easy to see that $f\in M(\mathbb R^n, \mathfrak M)$. Now, since $M$ is real maximal, there is a real number $p_i\in\mathbb R$ such that $M(f)=M(p_i)$. This implies that $Z(f-p_i)\in Z[M]$, but 
$$Z(f-p_i)=\underbrace{\mathbb R\times ...\times\mathbb R\times \{p_i\}\times\mathbb R\times ...\mathbb R}_\text{ n factors}.$$
It follows then  $\{p\}=\{(p_1,..., p_n)\}=\bigcap_{i=1}^{n}Z(f-p_i)\in Z[M]$ which means that  $M$ is fixed and hence $(\mathbb R^n, \mathfrak M)$  is realcompact.
\end{example}
In the following, we  provide a characterization of real  maximal ideals of $M(X, \A)$. First note that, by Lemma \ref{multiple elements} and the argument after that, it is easy to see that Theorem 3.7 in \cite{ABS} holds in this framework, as its proof carries over without modification. This yields the following lemma which is a counterpart of \cite[Theorem 7.6]{GJ} and holds in $M(X, \A)$, as its proof applies equally well.
Note that, by Corollary \ref{beta X}, every   $f\in M(X, \A)$ has   a continuous extension $ f^* : (\beta _{\A}X, \tau_{\bar\A}) \to \mathbb R^*$, where $\mathbb R^*=\mathbb R\cup \{\infty\}$ is the one-point compactification of $\mathbb R$. 
\begin{lemma}\label{value of f*}
Let $(X, \A)$ be  an elementary space and $f\in M(X, \A)$, then the following statements hold
\begin{enumerate}
\item $ f^*(p)=\infty$ if and only if $|M^p(f)|$ is infinitely large.
\item $ f^*(p)=r$ if and only if $|M^p(f)-r|$ is either infinitely small or zero.
\end{enumerate}
\end{lemma}
The following corollary shows that, $\upsilon_{\A}X$ is the largest subspace of $\beta_{\A}X$ in which every $f\in M(X, \A)$ has a real-valued continuous extension.   

\begin{corollary}\label{Theorem 8.4}
Let $(X, \A)$ be an elementary space and $p\in\beta_{\A}X$. Then $M^p$ is a real maximal ideal  if and only if $ f^*(p)\ne\infty$, for every $f\in M(X, \A)$.
\end{corollary}
\begin{proof}
The necessity is due to Lemma \ref{value of f*}(1). To prove the sufficiency, suppose that the condition hold and define $\phi : M(X, \A)\to \mathbb R$ by $\phi(f)=f^*(p)$. Then, this is a well-defined ring epimorphism, where $\text{ker}(\phi)=M^p$. Thus $M(X, \A)/M^p\simeq\mathbb R$ and $M^p$ is real.
\end{proof}

Let $\A$ and $\A'$ are two algebras in a set $X$ such that $A'\subseteq \A$. Then the elementary homomorphism $i :(X, \A)\to (X, \A')$ extends to an elementary homomorphism $\bar i :\beta_{\A}X\to \beta_{\A'}X$, by Theorem \ref{m-homo extension}. Using the proof of the same theorem, whenever $M^{p}$ and $M^{\bar i(p)}$ are the maximal ideals of $M(X, \A)$ and $M(X, \A')$, respectively, then $M^{\bar i(p)}=M^p\cap M(X, \A')$. Using this, we observe that $\bar i(\upsilon_{\A}X)\subseteq\upsilon_{\A'}X$. Indeed, if $p\in \upsilon_{\A}X$ then, by the third isomorphism theorem for modules \cite[6.38]{S}, we have
$$\mathbb R\subseteq \frac{M(X, \A')}{M^{\bar i(p)}}=\frac{M(X, \A')}{M^p\cap M(X, \A')}\simeq \frac{M(X, \A')+M^p}{M^p}\subseteq \frac{M(X, \A)}{M^p}\simeq\mathbb R.$$
This implies that $\frac{M(X, \A')}{M^{\bar i(p)}}\simeq \mathbb R$, i.e., $M^{\bar i(p)}$ is real or equivalently $\bar i(p)\in\upsilon_{\A'}X$. This leads to the following characterization of real maximal ideals of $M(X, \A)$. First, remember that a filter $\mathcal F$ on a set $X$ is said to have the countable intersection property if for any countable subfamily $\{F_n\}_{n=1}^\infty$ of $\mathcal F$, we have $\bigcap_{n=1}^\infty F_n\ne\emptyset$.
\begin{theorem}\label{real ideals}
 Let $(X, \A)$ be an elementary space and $M^p$ ($p\in\beta_{\A}X$) be a maximal ideal of $ M(X, \A)$. Then $M^p$ is real if and only if $Z[M^p]$ has the countable intersection property.
\end{theorem}
\begin{proof}
Sufficiency. Suppose, on the contrary, that $M^p$ is not a real maximal ideal, i.e., $p\notin\upsilon_{\A}X$. Thus, by Corollary \ref{Theorem 8.4}, there is a $\A$-measurable function $f\in M(X, \A)$ such that $f^*(p)=\infty$. It follows then, $A_n=\{x\in X : |f(x)|\ge n\}\in Z[M^p]$ and clearly $\bigcap_{n=1}^\infty A_n=\emptyset$, a contradiction.

Necessity. Suppose, on the contrary, that  there is a countable subfamily $\{A_n\}_{n=1}^\infty$ of $Z[M^p]$, such that $\bigcap_{n=1}^\infty A_n=\emptyset$. Without loss of generality, we may assume that $A_{n+1}\subseteq A_n$, for any $n\in\mathbb N$. This implies that for any infinite subsequent $I$ of $\mathbb N$, we have $\bigcap_{n\in I} A_n=\emptyset$. Now, it is easy to see that the family
$$\A'=\{A_n\}_{n=1}^\infty\cup \{A^c_n\}_{n=1}^\infty\cup\{X, \emptyset\},$$
is a $\sigma$-algebra in $X$ and $\A'\subseteq \A$. Thus, the argument preceding the theorem implies that $\bar i(p)\in\upsilon_{\A'}X$ (i.e., $M^{\bar i(p)}$ is real),  where $\bar i$ is the elementary extension of $i : (X, \A)\to (X, \A')$. In addition, using the same argument, we have $M^{\bar i(p)}=M^p\cap M(X, \A')$ which implies that $\{A_n\}_{n=1}^\infty\subseteq Z[M^{\bar i(p)}]$. This, contradicts the fact that $M^{\bar i(p)}$ is real maximal, as $\bigcap_{n\in I} A_n=\emptyset$, see \cite[Theorem 3.9]{ABS}.
\end{proof}

The following proposition is the counterpart of Theorem \ref{Extension into compact}, for extending $\A$-measurable mappings to $\upsilon_{\A}X$.
\begin{proposition}\label{Extension into realcompact}
Let $(X, \A)$ be an elementary space and $\varphi$ be an $\A$-measurable mapping  from $(X, \A)$ into a  realcompact topological space $Y$. Then $\varphi$ extends to  a continuous mapping $ \varphi^\circ: \upsilon_{\A}X\to Y$. Moreover, if  $\A$ is a $\sigma$-algebra and $Y$ is second countable, then the extended mapping $\varphi^\circ$  is $\bar\A_R$-measurable.
\end{proposition}
\begin{proof}
The hypothesis implies that $ \varphi: X\to Y\subseteq \beta Y$ is $\A$-measurable and thus it extends to a continuous mapping $\bar\varphi :\beta_{\A}X\to \beta Y$, by Theorem \ref{Extension into compact}. Let $\varphi^\circ=\bar\varphi|_{\upsilon_{\A}X}$. Now, if $p\in\upsilon_{\A}X$ then $\A^p=Z[M^p]$ has the countable intersection property, by Theorem \ref{real ideals}. Thus, $\varphi^{\#}\A^p$ (see the proof of Theorem \ref{Extension into compact}) has the countable intersection property as well. It follows from this and \cite[Lemma 8.12]{GJ}, that $\varphi^{\#}\A^p$ is fixed and hence $\bar\varphi(p)\in Y$, i.e., $\varphi^\circ(p) :\upsilon_{\A}X\to Y$ is well-defined.

To prove the last part, first  note that  $ \bar\A_{R}$ is a $\sigma$-algebra in $\upsilon_{\A}X$, whenever $(X, \A)$ is a measurable space, by \cite[Theorem 2.4]{AABS}. Next, since $Y$ is completely regular, the family  $\mathcal B$ consisting of zero-sets with non-empty interior constitute a base for the open sets, by \cite[3.2(b)]{GJ}. Now, using second countability of $Y$, there is a countable subfamily $\mathcal B_0$ of $\mathcal B$ which is a base for the open sets, by \cite[Theorem 1.1.15]{E}. Thus, if $U$ is an open subset of $Y$, for any point $p\in U$, there exists a zero-set-neighborhood $Z_p\in\mathcal B_0$ of $p$ which is contained in $U$. This implies that 
$$U=\bigcup_{n=1}^{\infty}Z_n=\bigcup_{n=1}^{\infty}\text{int}_Y Z_n,$$
for some countable subfamily $\{Z_n\}_{n=1}^{\infty}$ of $\mathcal B_0$. The rest of the proof is similar to the proof of Theorem \ref{Extension into compact}.
\end{proof} 

\begin{corollary}\label{uX properties}
Every elementary space $(X, \A)$, has a realcompact extension $\upsilon_{\A}X$ which is contained in $\beta_{\A}X$ and has the following properties:
\begin{enumerate}
\item $\upsilon_{\A}X$ is the largest subspace of $\beta_{\A}X$ in which every $f\in M(X, \A)$ has a continuous extension $f^{\upsilon_{\A}}\in  C(\upsilon_{\A}X)$. Moreover, if $\A$ is a $\sigma$-algebra, then $f^{\upsilon_{\A}}\in M(\upsilon_{\A}X, \bar\A_R)\subseteq C(\upsilon_{\A}X, \tau_{\bar\A_R})$.
\item $(\upsilon_{\A}X, \bar\A_R)$ is the smallest realcompact  extension of $X$.
\item If $X\subseteq T\subseteq\upsilon_{\A}X$ and $\A'=\bar \A|_T$, then $\upsilon_{\A'}T=\upsilon_{\A}X$.
\item If $\A$ is a $\sigma$-algebra, $X\subseteq T\subseteq \upsilon_{\A}X$ and $\{A_n\}_{n\in\mathbb N}\subseteq\A$, then
$$\text{cl}_T\bigcap_{n\in\mathbb N}A_n=\bigcap_{n\in\mathbb N}\text{cl}_T A_n.$$
\item Every $\A$-measurable mapping  from $X$ into any realcompact topological space $Y$, extends continuously to   $\upsilon_{\A}X$. 
\item Every  elementary homomorphism $\varphi$ from $(X, \A)$ into any realcompact elementary space $(Y, \tilde A)$,  extends to an elementary homomorphism on $(\upsilon_{\A}X, \tau_{\bar\A_R})$. 
\end{enumerate}
\end{corollary}
\begin{proof}
First, we prove that $(\upsilon_{\A}X, \bar\A_R)$ is realcompact. To do this, let $\mathcal F$ be a real $\bar\A_R$-ultrafilter on $\upsilon_{\A}X $ and define $\mathcal F_0=\{A\cap X : A\in\mathcal F\}$. Since for any $A\in\A$, we have $\upsilon_{\A}X\setminus\text{cl}_{\upsilon_{\A}X}A=\text{cl}_{\upsilon_{\A}X}(X\setminus A)$, it is easy to see that $\mathcal F_0$ is an $\A$-ultrafilter on $X$; see also Lemmas \ref{m-sets of extensions} and \ref{A-ultrafilter}(2). On the other hand, by Theorem \ref{real ideals},  if $\{A_n\}_{n=1}^\infty\subseteq \mathcal F$ then there is $p\in\bigcap_{n=1}^\infty A_n\subseteq\upsilon_{\A}X$. This implies that $\{A_n\cap X\}_{n=1}^\infty \subseteq \A^p$, where $\A^p$ is the unique real $\A$-ultrafilter on $X$, obtained by Corollary \ref{Extensions}, that is convergent to $p\in\upsilon_{\A}X$. Another application of Theorem \ref{real ideals} shows that, $\bigcap_{n=1}^\infty A_n\cap X\ne\emptyset$, i.e., $\mathcal F_0$ is real. This means that $\mathcal F_0$ is convergent to a point $p_0\in\upsilon_{\A}X$. Therefore $p_0\in\bigcap_{\A\in\mathcal F}A$ and $\mathcal F$ is fixed,  as $\text{cl}_{\upsilon_{\A}X}(A\cap X)=A$, for any $A\in\bar\A_R$.
\begin{enumerate}
\item This is due to Corollary \ref{Theorem 8.4}. The last part, is a special case of Proposition \ref{Extension into realcompact}, as  $\mathbb R$ is second countable.
\item  Let $(T, \A')$ be a realcompact elementary    extension of $X$  and  $p\in \upsilon_{\A}X$. Moreover, let $\A^p$ be the unique real $\A$-ultrafilter on $X$, obtained by Corollary \ref{Extensions}, that  converges to $p$ and define $\A^{'p}:=\{ A'\in \A' :  A'\cap X\in\A^p \}$. A straightforward argument, using Lemmas \ref{m-sets of extensions} and \ref{A-ultrafilter}(2), shows that $\A^{'p}$ is an $\A'$-ultrafilter on $T$. On the other hand, since $\A^p$ has the countable intersection property, by Theorem \ref{real ideals},  if $\{A'_n\}_{n=1}^\infty \subseteq \A^{'p}$, then 
$$\emptyset\ne\bigcap_{n=1}^\infty (A'_n\cap X)=\bigcap_{n=1}^\infty A'_n\cap X,$$
which implies that $\bigcap_{n=1}^\infty A'_n\ne\emptyset$, i.e., $\A^{'p}$ is real by the same theorem. Now, as $T$ is realcompact,  $\A^{'p}$ is fixed and it is easy to see that $ \A^{'p}$ is convergent (in $\beta_{\A}X=\beta_{\A'}T$) to $p$ and thus $p\in T$, i.e., $\upsilon_{\A}X\subseteq T$.

\item Since $T\subseteq\upsilon_{\A}X$ and the latter  is  realcompact by the first argument, $\upsilon_{\A'}T\subseteq \upsilon_{\A}X$, by part (2). On the other hand, since $\upsilon_{\A'}T$ is a realcompact extension of $X$, $\upsilon_{\A}X\subseteq \upsilon_{\A'}T$, by part (2).
\item  Let $p\in \bigcap_{n\in\mathbb N}\text{cl}_T A_n$ and $\A^p$ be the unique real $\A$-ultrafilter on $X$, obtained by Corollary \ref{Extensions}, convergent to $p$. 
Using the same corollary,     $A_n\in \A^p$ for every $n\in\mathbb N$. Now, since   $\A^p$ is closed under countable intersection, by \cite[Theorem 3.9]{ABS}, we conclude that $\bigcap_{n\in\mathbb N}A_n\in\A^p$. This means that $p\in \text{cl}_T\bigcap_{n\in\mathbb N}A_n$, by Corollary \ref{Extensions}. As the reverse inclusion is evident, we are done.
\item This is duo to Proposition \ref{Extension into realcompact}.
\item This is done, by applying the same proof of Theorem \ref{m-homo extension}. Only note that, $\varphi^{\#}\A^p$ has the countable intersection property, whenever $\A^p$ is so, see Theorem \ref{real ideals}.
\end{enumerate}
\end{proof}
In fact, the mapping $f \to f^{\upsilon_{\A}}$ in part (1) of the preceding corollary, is  an (ring) isomorphism from $M(X, \A)$  onto $M(\upsilon_{\A}X, \bar\A_R)$, for a measurable space $(X, \A)$.

The following result, establishes a necessary condition for  an elementary space $(X, \A)$ to be realcompact.

\begin{proposition}\label{real measurable is realcompact}
If an elementary space $(X, \A)$  is realcompact, then $(X, \tau_{\A})$ is a realcompact topological space.
\end{proposition}
\begin{proof}
 Using \cite[Lemma 8.12]{GJ} a topological space $X$ is realcompact if and only if every prime $z$-filter which is closed under countable intersection is fixed. Note that $M(X, \A)\subseteq C(X, \tau_{\A})$, which implies that $\A\subseteq Z(X)$. Now, suppose that $\mathcal F$ is a prime $z$-filter on $X$ which has the countable intersection property and let $\mathcal F_{\A}=\mathcal F\cap \A$. Since $\A\subseteq Z(X)$, it is clear that  $\mathcal F_{\A}$ is a  $\A$-ultrafilter, by Lemma \ref{A-ultrafilter}(2), and it has the countable intersection property.  Thus, Theorem \ref{real ideals} implies that $\mathcal F_{\A}$ is a real $\A$-ultrafilter on $X$ and hence it is fixed. We assume that $p\in\bigcap\mathcal F_{\A}$ and show that $p\in\bigcap\mathcal F$. Suppose, on the contrary, that $p\notin Z$, for some $Z\in\mathcal F$. Since $Z\subseteq X$ is closed and $\A$ is a base for the open (closed) subsets of $(X, \tau_{\A})$, there exists some $A\in\A$ such that $p\in A\subseteq X\setminus Z$. Then $Z\subseteq X\setminus A$ and therefore $X\setminus A\in\mathcal F$, as $X\setminus A\in\A\subseteq Z(X)$. But, this implies that $X\setminus A\in\mathcal F_{\A}$ and $p\notin X\setminus A$,  which is  a contradiction. This  shows that $\mathcal F$ is fixed and hence $(X, \tau_{\A})$ is realcompact by \cite[Lemma 8.12]{GJ}.
\end{proof}
The   convers of  Proposition \ref{real measurable is realcompact} is not  true in general, even if $(X, \A)$ is a measurable space;  see the following example.
\begin{example}\label{realcompact&realmeasurable}
Let $\omega_1$ be the first uncountable ordinal and $X=\mathbb W=\{\alpha : \alpha<\omega_1 \}$ be the set of all countable ordinals.  For any ordinal number $\alpha$, the set of all ordinals less than $\alpha$ is denoted by $W(\alpha)$, in particular $\mathbb W=W(\omega_1)$. For any ordinal number $\alpha$, $\mathbb W\setminus W(\alpha)$ is called a tail in $\mathbb W$; see \cite[5.12]{GJ}, for more details on the space of (countable) ordinals.  Now, let 
$$\A=\{ A\subseteq X : |A|\le\aleph_0 \text{ or } |A^c|\le\aleph_0\},$$
where  $\aleph_0$ is the least infinite cardinal number. It is easy to see that $\A$ is a $\sigma$-algebra in $X$. Obviously, $\tau_{\A}$ is the discrete topology on $X$ and since $|X|=\mathfrak c$ is a nonmeasurable cardinal, $(X, \tau_{\A})$ is a realcompact topological space by the argument after \cite[Corollary 8.18]{GJ}. However
$$\A^{\omega_1}=\{A\in\A : A \text{ contains a tail in } \mathbb W\},$$
is a free real $\A$-ultrafilter in $X$. In fact, if $A_0\in\A$ and  $A_0\notin\A^{\omega_1}$, then $A^c$ contains a cofinal subset of $\mathbb W$ and since no countable set in $\mathbb W$ is cofinal by \cite[5.12(a)]{GJ}, we conclude that $A_0^c$ is uncountable. This shows that $A_0^c\in \A^{\omega_1}$; otherwise, a similar argument implies that $A_0$  is uncountable, which contradicts the definition of measurable sets of $\A$. Moreover, another application of \cite[5.12(a)]{GJ} implies that  $\A^{\omega_1}$ is closed under countable intersection.  Hence $\A^{\omega_1}$ is a real $\A$-ultrafilter by Lemma \ref{A-ultrafilter}(2) and \cite[Theorem 3.9]{ABS}, which is clearly free.

As a matter of fact, this is the only free real $\A$-ultrafilter in $X$ and $\upsilon_{\A}X=\mathbb W^*=W(\omega_1+1)$ and
$$\bar\A_R=\{A\subseteq X : |A|\le\aleph_0 \}\cup\{A\cup\{\omega_1\} :      A\subseteq X \text{ and } |A^c|\le\aleph_0\}.$$ 
 Indeed, if $\A^q$ is a distinct  real $\A$-ultrafilter in $\mathbb W$, then there are disjoint measurable sets $A$ and $B$ of $\A$ such that $A\in \A^{\omega_1}$ and $B\in \A^q$, by Lemma \ref{A-ultrafilter}(3). This implies that $B\subseteq W(\alpha)$, for some $\alpha<\omega_1$, and hence $W(\alpha)\in\A^q$. Now, since $W(\alpha)=\bigcup_{\sigma<\alpha}\{\sigma\}$ and $\A^q$  is closed under countable intersection by \cite[Theorem 3.9]{ABS}, there exists $\sigma<\alpha$ such that $\{\sigma\}\in\A^q$, i.e., $\A^q$ is fixed.
\end{example}
In the previous section we provide  equivalent conditions (in terms of elementary sets) under which $\beta X=\beta_{\A}X$. Now, since it is easy to see that an elementary space $(X, \A)$ is realcompact if and only if $(X, \tau_{\A})$ is realcompact and $ \upsilon X=\upsilon_{\A}X$, this leads to the following question.
\begin{question}
Under what conditions on an algebra $\A$, do we have $ \upsilon X=\upsilon_{\A}X$?
\end{question}

 Suppose that $(X, \tau)$ is a topological space. Then we may consider  $(X, \mathscr B)$ or $(X, \A_z)$, where  $\mathscr B=<\tau>$ is the $\sigma$-algebra generated by open sets and $\A_z=<Z(X)>$. The members  of $\mathscr B$ are called the Borel sets of $X$.
In \cite{ABS}, the authors raised a question that if $(X, \tau)$ is a realcompact topological space, is $(X, \mathscr B)$  a realcompact measurable space? The following proposition  provides conditions under which  the  measurable space $(X, \mathscr B)$ is  realcompact.
 First, we need the following lemma which describes measurable sets of a $\sigma$-algebra  generated by a nonempty family $\mathcal B$ of subsets of a set $X$.  
\begin{lemma}
 Let $\mathcal B$ be a nonempty family  of subsets of a set $X$.  Then  $<\mathcal B>=\bigcup_{\alpha<\omega_1}\mathcal B^\alpha$, where $\mathcal B^0=\mathcal B \cup \mathcal B^c$ and  for any countable ordinal number $\beta>0$, 
$$\mathcal B^\beta=\{\bigcup_{n=1}^\infty B^{\alpha_n}_n : B^{\alpha_n}_n\in \mathcal B^{\alpha_n} \text{, } \alpha_n<\beta\}\cup \{\bigcap_{n=1}^\infty B^{\alpha_n}_n : B^{\alpha_n}_n\in \mathcal B^{\alpha_n} \text{, } \alpha_n<\beta\},$$
where $\omega_1$ is the first uncountable ordinal number and $\mathcal B^c=\{B^c : B\in\mathcal B\}$.
\end{lemma}
\begin{proof}
Using \cite[5.12(a)]{GJ}, every countable family of countable ordinal numbers has a supremum less than $\omega_1$ and thus
 $\bigcup_{\alpha<\omega_1}\mathcal B^\alpha$ is closed under countable union. Now, it is enough to show that it is closed under complement operation, as in this case clearly $\bigcup_{\alpha<\omega_1}\mathcal B^\alpha$ is the smallest $\sigma$-algebra which contains $\mathcal B$ and then we are through. We prove this by the principle of transfinite induction. Clearly, for any $B\in \mathcal B^0$, we have $B^c\in \mathcal B^0\subseteq \mathcal B^1$. Now, suppose that for every ordinal number $\alpha<\beta$, if $B\in \mathcal B^\alpha$ then $B^c\in \mathcal B^{\alpha+1}$. Let $B=\bigcup_{n=1}^\infty B^{\alpha_n}_n\in \mathcal B^\beta$, where $B^{\alpha_n}_n\in\mathcal B^\alpha_n$ and $\alpha_n<\beta$. Thus by the inductive assumption, $X\setminus B^{\alpha_n}_n\in \mathcal B^{\alpha+1}$ and therefore $B^c=\bigcap_{n=1}^\infty X\setminus B^{\alpha_n}_n\in\mathcal B^{\alpha_0}\subseteq \mathcal B^{\beta+1}$, where $\alpha_0=sup\{\alpha_n : n\in\mathbb N\}\le\beta$. Similarly, if $B=\bigcap_{n=1}^\infty B^{\alpha_n}_n\in\mathcal B^\beta$, then $B^c\in\mathcal B^{\beta+1}$ and thus we are done.
\end{proof}
\begin{proposition}
Let $(X, \tau)$ be  a realcompact  topological space then  the following statements hold.
\begin{enumerate}
\item $(X, \A_z)$ is realcompact.
\item  If every closed set in $X$ is a $G_{\delta}$-set, then $(X, \mathscr B)$ is realcompact.
\end{enumerate}
\end{proposition}
\begin{proof}
\begin{enumerate}
\item Let $\mathcal F$ be a real $\A_z$-ultrafilter on $X$ and define 
$$ \mathcal F_z:=\{Z\in Z(X) : Z\in\mathcal F \}=\mathcal F\cap Z(X).$$
First, since $Z(X)\subseteq \A_z$ and zero-sets are closed under countable intersection, it is easy to see that $\mathcal F_z$ is a prime  $z$-filter on $X$ which is closed under countable intersection and thus it is fixed, by \cite[Lemma 8.12]{GJ}.

 Next, we show that $\mathcal F_z$ is a $z$-ultrafilter. To this end, suppose that  $Z_0=Z(f)\in Z(X)$ and $Z_0\cap Z\ne\emptyset$, for every $Z\in\mathcal F_z$. We claim that $Z_n=f^{-1}[\frac{-1}{n}, \frac{1}{n}]\in \mathcal F_z$, for every $n\in\mathbb N$. If it is not the case, then $C=X\setminus f^{-1}[\frac{-1}{m}, \frac{1}{m}]\in  \mathcal F$, for some $m\in\mathbb N$; see Lemma \ref{A-ultrafilter}(2). It follows that, the zero-set $D=X\setminus f^{-1}(\frac{-1}{m+1}, \frac{1}{m+1})$ belongs to  $\mathcal F$ and hence to $\mathcal F_z$, as $C\subseteq D$. But $Z_0\cap D=\emptyset$, which is a contradiction. This proves the claim, and since $\mathcal F_z$ is closed under countable intersection, then $Z_0=\bigcap_{n=1}^\infty Z_n\in \mathcal F_z$. Therefore, $\mathcal F_z$ is a $z$-ultrafilter, by the argument after \cite[Theorem 2.6]{GJ}. Now, as $\mathcal F_z$ is a fixed $z$-ultrafilter, if   $x\in\bigcap\mathcal F_z$ then a zero-set $Z$, belongs to $\mathcal F_z$ if and only if $x\in Z$.

 Finally, we show that $x$ belongs to any member of $\mathcal F$, which means that $\mathcal F$ is fixed and then we are through. Using the preceding lemma, $\A_z=\bigcup_{\alpha<\omega_1}\mathcal B^\alpha$, thus we prove the last statement by the principle of transfinite induction, as $\mathcal F=\bigcup_{\alpha<\omega_1}\mathcal F\cap\mathcal B^\alpha$. To show that $x$ belongs to every member of $\mathcal F\cap \mathcal B^0$, clearly it is enough to prove that it  belongs to every cozero-set of $X$, in $\mathcal F$. To let this end, suppose that  $D=X\setminus Z\in\mathcal F$, for some zero-set $Z\in Z(X)$. This implies that $Z\notin\mathcal F_z$ and since $\mathcal F_z$ is a $z$-ultrafilter, we conclude that $x\notin Z$, i.e., $x\in D$.

 Now, suppose that $x$ belongs to every member of $\mathcal F\cap \mathcal B^\alpha$, for any ordinal number $\alpha<\beta$. If $B=\bigcup_{n=1}^\infty B^{\alpha_n}_n\in\mathcal F\cap\mathcal B^\beta$, then $B^{\alpha_m}_m\in\mathcal F$, for some $m\in\mathbb N$, as $\mathcal F$ is closed under countable intersection, by \cite[Theorem 3.9]{ABS}. Thus $x\in B^{\alpha_m}_m\subseteq B$, by the inductive assumption. On the other hand, if $B=\bigcap_{n=1}^\infty B^{\alpha_n}_n\in\mathcal B^\beta$, then $B^{\alpha_n}_n\in\mathcal F$ for every $n\in\mathbb N$, as $\mathcal F$ is an $\A_z$-filter. Hence $x\in \bigcap_{n=1}^\infty B^{\alpha_n}_n=B$, by the inductive assumption and we are done.
\item  Let $\mathcal F$ be a real $\mathscr B$-ultrafilter on $X$ and define $ \mathcal F_z=\mathcal F\cap Z(X)$. 
Since every zero-set is a Borel set,  a similar argument as in part (1), shows that $\mathcal F_z$ is a fixed real $z$-ultrafilter. Thus, if $x\in \bigcap\mathcal F_z$, a zero-set $Z\in Z(X)$ belongs to $\mathcal F_z$ (and hence to $\mathcal F$) if and only if $x\in Z$. Now, using this fact and complete regularity of $X$ which turns the zero-sets of $X$ into a base for the closed sets of $X$ (by \cite[Theorem 3.2]{GJ}) we conclude that $x$ belongs to every closed  member  of $\mathcal F$. On the other hand, if $U$ is an open subset of $X$, then by the assumptions  $U=\bigcup_{n=1}^\infty F_n$   is an $F_\sigma$-set, i.e., a countable union of closed sets. Therefore, by the fact that $\mathcal F$ is closed under countable intersection, if $U\in\mathcal F$ then there is some $m\in\mathbb N$ such that $F_m\in\mathcal F$ and hence $x\in F_m\subseteq U$. This implies that  $x$ belongs to every member of $ \mathcal F\cap \mathcal B^0$, where $\mathcal B^0=\tau\cup\tau^c$; see the preceding lemma. The rest of the proof is similar to the argument of part (1). 
\end{enumerate}
\end{proof}
Now, using \cite[5.5.10(b)]{E} and the previous proposition we deduce the following result which provides a wide class of elementary realcompact spaces.
\begin{corollary}
Every paracompact space $X$ (in particular every metric space) with a nonmeasurable cardinal whose every closed subset is a $G_\delta$-set, is a realcompact measurable space with respect to the  Borel sets.
\end{corollary}

\section{Maximal ring of quotients of rings of elementary functions}
In this section, we provide a  description of the maximal ring of quotients of $M(X, \A)$ similar to that of the ring of continuous functions on a topological space. In particular we show that, the maximal ring of quotients of $M(X, \A)$ and $C(X)$ coincide, whenever $(X, \A)$ is a measurable space.
 
Let $R$ be a commutative   semi-prime ring, i.e., it has no nilpotent element except $0$. Recall from \cite[Theorem 1.5]{FGL}, that   a ring $S\supseteq R$ is called  a ring of quotients or a rational extension of $R$ provided that, for any  $0\ne b\in S$ there exists $a\in R$  such that $0\ne ba \in R$. Also, a ring is called {\it rationally complete} if it has no proper rational extension. Any commutative ring $R$  has a largest ring of quotients, named the maximal ring of quotients of $R$ which is denoted by $\mathcal Q(R)$, and is  rationally complete; see \cite[Corollary 1.9]{FGL}.

In the case of the ring of continuous functions, $C(X)$, on a completely regular Hausdorff space $X$, it has been shown that $\mathcal Q(C(X))$, which is denoted by $Q(X)$, is the direct limit of rings of continuous functions on dense open subsets
$$\mathcal Q(C(X))=\lim_{\substack{\longrightarrow \\ V\in\mathcal V_0(X)}}C(V),$$
where $\mathcal V_0(X)$ is the family of all dense open subsets of $X$; see \cite[Theorem 2.6]{FGL}. The direct limit, $\lim_{\longrightarrow \\ V\in\mathcal V_0(X)}C(V)$, may be thought of as the free union $\bigcup_{ V\in\mathcal V_0(X)} C(V)$ along with an equivalence relation, where $f_1\in C(V_1)$ and $f_2\in C(V_2)$ identified whenever $f_1$ and $f_2$ agrees on $V_1\cap V_2$. The following proposition, introduces some rings of quotients of $M(X, \A)$. First, remember that a function $f\in C(X)$ is called locally constant if for any $x\in X$ there is an open neighborhood $V$ of $x$ on which $f$ is constant.

\begin{proposition}\label{M is quotient of M*}
Let $(X, \A)$ be an elementary space. Then the following statements hold.
\begin{enumerate}
\item $M(X, \A)$ is a ring of quotients of $M^*(X, \A)$.
\item $M(X, \A)$ and $M^*(X, \A)$ have the same maximal ring of quotients. 
\item For any dense open subspace $V$ of $(X, \tau_{\A})$
$$C_L(V):=\{f\in C(V) : f \text{ is locally constant} \},$$
is a ring of quotients of $M(X, \A)$.
\end{enumerate}
\end{proposition}
\begin{proof}
\begin{enumerate}
\item If $0\ne g\in M(X, \A)$, then $1+g^2$ is a unit element of $M(X, \A)$, by Lemma \ref{units}. Hence $f=1/(1+g^2)\in M^*(X, \A)$ and $0\ne gf\in M^*(X, \A)$.
\item This part follows from (1).
\item First note that, it is easy to see that $C_L(V)$ is a subring of $C(V)$. Next,
 since for any $f\in M(X, \A)$ and  $r\in\mathbb R$, $f^{-1}(r)\in\A$, then $M(X, \A)\subseteq C_L(V)$. Now, let $0\ne f\in C_L(V)$. Thus there is $p\in V$ such that $f(p)=r\ne0$. Since  $V\subseteq X$ is a open, there is an elementary set $A_0\in \A$ such that $p\in A_0\subseteq V$. On the other hand, there is $A_1\in\A$ such that $f|_{A_1}=r$, as $f$ is locally constant. Thus, $0\ne f\chi_{A_0\cap A_1}=r\chi_{A_0\cap A_1}\in M(X, \A)$, i.e., $C_L(V)$ is a ring of quotients of $M(X, \A)$.
\end{enumerate}
\end{proof}
Before  giving the description of the maximal ring of quotients of $M(X, \A)$, we recall from \cite[1.3]{FGL}, that an ideal $D$ of a ring $R$ is said to be (rationally) dense if its only annihilator in $R$ is $0$. Also, from \cite[Corollary 1.9]{FGL}, for any ring $R$, we have 
$$\mathcal Q(R)=\lim_{\substack{\longrightarrow \\ D\in\mathcal D_0(R)}}\text{Hom\,}D,$$
where $\mathcal D_0(R)$ is the family of all dense ideals of $R$ and $\text{Hom\,}D$ is the ring of all $R$-module homomorphisms  $\varphi :D\to R$. In the case of rings of continuous functions, an ideal $D$ of $C(X)$ is dense if and only if $\text{coz\,}D:=\bigcup_{f\in D}\text{coz\,}f$ is a dense (open) subset of $X$; see \cite[Theorem 2.2]{FGL}. Indeed, the same fact holds in $M(X, \A)$, as the same proof of the aforementioned theorem carries over without modification. We need the following lemma as well.
\begin{lemma}\label{Hom}
If $D$ is a dense ideal of $M(X, \A)$, then $\text{Hom\,}D\subseteq C_L(\text{coz\,}D)$.
\end{lemma}
\begin{proof}
Let $\varphi\in\text{Hom\,}D$. If $x\in\text{coz\,}D=:V$, choose an elementary function $d\in D\subseteq M(X, \A)$ such that $d(x)\ne0$ and define
$$g(x)=\frac{\varphi(d)(x)}{d(x)}.$$
Whenever $d'\in D$ is another elementary function for which $d'(x)\ne0$, since $\varphi$ is a $M(X, \A)$-module homomorphism, then $d'\varphi(d)=d\varphi(d')$. This shows that, $g$ is a well-defined real-valued function on $\text{coz\,}D$. To prove that $g\in C_L(V)$, let $p\in V$ and suppose that $g(p)=\frac{\varphi(d)(p)}{d(p)}=r_0$, for some suitable $d\in D$. Moreover let, $A_0=\text{coz\,}d$, $A_1=d^{-1}(r_1)$ and $A_2=\varphi(d)^{-1}(r_0r_1)$, where $r_1=d(p)\ne0$. Since both  $d$ and $\varphi(d)$  belong to $M(X, \A)$, we conclude that $A_0, A_1$ and $ A_2$ are elementary sets, see the argument before Lemma \ref{units}. Hence, $p\in A_0\cap A_1\cap A_2\subseteq V$ and $g(x)=\frac{\varphi(d)(x)}{d(x)}=r_0$, for any $x\in A_0\cap A_1\cap A_2$.
Now, if $d\in  D$, for any $x\in V$ and a suitable choice of $d'\in D$ we have
$$\varphi(d)(x)=\frac{\varphi(d')(x)}{d'(x)}d(x)=g(x). d(x),$$
which means $\varphi(d)=g.d$. This  shows that $\varphi=g\in C_L(V)$.
\end{proof}
Now,  Proposition \ref{M is quotient of M*}, Lemma \ref{Hom} and the argument before that, yields the following description of $\mathcal Q(M(X, \A))$.
\begin{corollary}\label{Q(M(X))}
Let $(X, \A)$ be an elementary space, then
$$\mathcal Q(M(X, \A))=\lim_{\substack{\longrightarrow \\ V\in\mathcal V_0(X)}}C_L(V),$$
where $\mathcal V_0(X)$ is the family of all dense open subsets of $X$.
\end{corollary}
Whenever $(X, \A)$ is a measurable space, Lemma \ref{Z(X)=B(X)} implies that $C_L(V)=C(V)$, for any dense open subset $V$ of $X$. Thus, in this case, the preceding corollary  led to a simple description of the maximal ring of quotients of $M(X, \A)$.

\begin{corollary}\label{Max quotient of M}
If $(X, \A)$ is a measurable space, then $\mathcal Q(M(X, \A))=\mathcal Q(X)$.
\end{corollary}

\begin{corollary}
If an algebra $\A$ in a set $X$ contains all  one-point subsets  of $X$, then $\mathcal Q((M(X, \A))=\mathbb R^X$. 
\end{corollary}
\begin{proof}
Using the hypothesis $\{ \{x\} : x\in X\}\subseteq \A$, which implies that $\tau_{\A}=\mathcal P(X)$. Thus, by  Corollary \ref{Q(M(X))}, we conclude that
$$\mathcal Q(M(X, \A))=\lim_{\substack{\longrightarrow \\ V\in\mathcal V_0(X)}}C_L(V)= C_L(X)=\mathbb R^X,$$
as the only dense open set in the discrete space $(X, \tau_{\A})$, is $X$ itself.
\end{proof}
One may think  that, for an elementary space $(X, \A)$,
 $$\mathcal Q((M(X, \A))=\lim_{\substack{\longrightarrow \\ V\in\mathcal V_0(X)}}M(V, \A|_V),$$
where $\mathcal V_0(X)$ is the family of all dense open sets in $X$. The preceding result shows that this is not the case. In fact, assume that $X$ is a set  and $\A$ is an algebra in $X$ which contains all the singletons, but $\A\ne \mathcal P(X)$. Then $\mathcal Q((M(X, \A))=\mathbb R^X$, whereas since the only dense open subset of $X$ is $X$ itself, 
$$\lim_{\substack{\longrightarrow \\ V\in\mathcal V_0(X)}}M(V, \A|_V)=M(X, \A),$$
which is clearly  not equal to $\mathbb R^X$.
\begin{remark}
Note that, by \cite[1.2]{FGL}, the classical ring of quotients of a ring $R$ (denoted by $\mathcal Q_{\text{cl}}(R)$) is the ring  of equivalence classes of  formal quotients $r/s$, where $r\in R$ and $s$ is a non-zero-divisor of $R$. If $R$ is a regular (von-Neumann regular) ring,    since    all its non-zero-divisors are unit, we have $\mathcal Q_\text{cl}(R)=R$. Now,  Lemma \ref{multiple elements} implies  that  $M(X, \A)$ is regular and  thus we conclude that $\mathcal Q_\text{cl}(M(X, \A))=M(X, \A)$.
\end{remark}

\section{When is $M(X, \mathcal A)$ a $C(Y)$?}
In this section, we turn to  the questions posed in \cite{ABS}. In the first one, the authors asked, if $(X, \A)$ is a measurable space, does there exist a (topological) $P$-space $Y$ such that $M(X, \A)\simeq C(Y)$? In the second one, they asked, under what conditions on a $\sigma$-algebra $\A$ in a set $X$, is the ring $M(X, \A)$  isomorphic with $C(X, \tau_{\A})$? The following proposition provides an answer to \cite[Question 5.5]{ABS}.

\begin{proposition}\label{M=C(X)}
Let $(X, \A)$ be an elementary space, then the following are equivalent.
\begin{enumerate}
\item $M(X, \A)\simeq C(X, \tau_{\A})$.
\item $\A$ is a $\sigma$-algebra and $\A\simeq Z(X)$, as Boolean algebras.
\end{enumerate}
\end{proposition}
\begin{proof}
(1) $\Rightarrow$ (2). Suppose that $\psi : M(X, \A)\to C(X, \tau_{\A})$ is a ring isomorphism. Then, $\psi$ takes real maximal ideals of $M(X,\A)$ to real maximal ideals of $C(X,\tau_{\A})$ and vice-versa. In other words,    $\psi$ defines a one to one  mapping $\varphi$ from $\upsilon_{\A}X$ onto $\upsilon X$ where for any $p\in\upsilon_{\A}X$,
$$M^{\varphi(p)}=\{f\in C(X) : \psi^{-1}(f)\in M^p\}.$$
Here, $M^p $ (resp., $M^{\varphi(p)}$) is the real maximal ideal of $M(X,\A)$ (resp., $C(X, \tau_{\A})$) corresponding to $p\in\upsilon_{\A}X$ (resp., $\varphi(p)\in\upsilon X$). Moreover, using \cite[Theorem 7.3]{GJ} and Corollary \ref{G-K theorem},  for any $f\in M(X,\A)$ we have
\begin{align*}
\varphi(\text{cl}_{\upsilon_{\A}X}Z(f))&=\{\varphi(p) : p\in \text{cl}_{\upsilon_{\A}X}Z(f)\} \\
&=\{\varphi(p) : f\in M^p \} \\
&=\{q\in\upsilon X : \psi(f)\in M^q\}=\text{cl}_{\upsilon X}Z(\psi(f)). \tag{1}
\end{align*}
Now, define $\phi: \A\to Z(X)$ by $\phi(A)=\varphi(\text{cl}_{\upsilon_{\A}X}A)\cap X=Z(\psi(f))$, where $A=Z(f)$ for some $f\in M(X, \A)$. Using (1),  $\phi$    is a well-defined  one to one mapping and it is clearly onto. Moreover, by   Corollary \ref{beta X}(5) and  since $\varphi$ is one to one, for any $A, B\in \A$, we have
\begin{align*}
\phi(A\cap B) &=\varphi(\text{cl}_{\upsilon_{\A}X}(A\cap B))\cap X\\
&=\varphi(\text{cl}_{\upsilon_{\A}X}A\cap \text{cl}_{\upsilon_{\A}X}B)\cap X\\
&=\varphi(\text{cl}_{\upsilon_{\A}X}A)\cap \varphi(\text{cl}_{\upsilon_{\A}X}B)\cap X\\
&= \phi(A)\cap \phi(B).
\end{align*}
A similar argument shows that $\phi(A\cup B)= \phi(A)\cup \phi(B)$ and $\phi(A^c)=\phi(A)^c$, as $\varphi$ is onto as well. This is show that $\phi$ is a Boolean algebra isomorphism, see also \cite[2.8]{W1}. Finally, using \cite[1.14(a)]{GJ},  $Z(X)$ is closed under countable intersection and  by \cite[Theorem 8.7(IV)]{GJ}, for any countable family of elementary sets $\{A_n\}_{n\in\mathbb N}$, we have 
$$\varphi(\bigcap_{n=1}^\infty\text{cl}_{\upsilon_{\A} X}A_n)=\bigcap_{n=1}^\infty\varphi(\text{cl}_{\upsilon_{\A} X}A_n)=\bigcap_{n=1}^\infty\text{cl}_{\upsilon X}Z_n=\text{cl}_{\upsilon X}Z_0,$$
where $Z_n=\phi(A_n)$ and $Z_0=\bigcap_{n=1}^\infty Z_n$. It follows then $\text{cl}_{\upsilon_{\A} X}A_0=\bigcap_{n=1}^\infty\text{cl}_{\upsilon_{\A} X}A_n$ or equivalently $A_0=\bigcap_{n=1}^{\infty}A_n$, where $A_0=\phi^{-1}(Z_0)\in\A$. This means that $\A$ is closed under countable intersections (unions), i.e., $\A$ is a $\sigma$-algebra.

(2) $\Rightarrow$ (1). Let $\phi: \A\to Z(X)$ be the Boolean isomorphism. Thus using Lemma \ref{A-ultrafilter}, \cite[Theorem 5.14]{GJ} and \cite[Theorem 3.9]{ABS}, it is easy to see that,  the image of every real $\A$-ultrafilter is a real $z$-ultrafilter. In other words, $\phi$ defines a one to one mapping $\varphi$ from $\upsilon_{\A}X$ onto $\upsilon X$ with the following property:
\begin{align*}
\varphi(\text{cl}_{\upsilon_{\A}X}A)=\text{cl}_{\upsilon X}\phi(A) \;\;\;\text{, for any } A\in\A. \tag{2}
\end{align*}
Now, define $\psi: C(\upsilon X)\to M(\upsilon_{\A}X, \bar\A_R)$ by $\psi(f)=f\circ\varphi$. Using property (2), $f\circ\varphi^{-1}(U)\in\bar\A_R$, for   any open subset $U$ of $\mathbb R$. Note that the zero-sets of  $\upsilon X$ are of the form $\text{cl}_{\upsilon X} Z$, for some $Z\in Z(X)$, by \cite[Remark 8.8(b)]{GJ}. This means that $\psi$ is well defined and since $\varphi$ is invertible, $\psi$ is one to one and onto. Moreover, clearly, $\psi$ is a ring homomorphism and therefore $C(\upsilon X)\simeq M(\upsilon_{\A}X, \bar\A_R)$. Now the result follows from \cite[Remark 8.8(a)]{GJ} and the argument after Corollary \ref{uX properties}.
\end{proof}
Now, we return to \cite[Question 5.3]{ABS} and provide conditions on an algebra $\A$ in a set $X$ such that $M(X, \A)$ is isomorphic to $C(Y)$, for some topological space $Y$. Note that, in particular, in this case $\A$ is a $\sigma$-algebra, see also the proof of the implication (1)$\Rightarrow$(2) in Proposition \ref{M=C(X)}.
First, recall from \cite{HJ}, that a $\phi$-algebra is an archimedean $f$-algebra with an identity element $1$, see Definition 1.5 in the same paper. If $A$ is a $\phi$-algebra, $\mathfrak M(A)$ denote the space of all maximal absolutely convex (or $l$-ideals) ideals of $ A$ equipped with the hull-kernel  topology, which is compact by \cite[Theorem 2.3]{HJ}. In addition, Theorem 2.3 in the same reference implies that every $\phi$-algebra $A$ is isomorphic to an algebra $\bar A$ of extended (real-valued) functions on $\mathfrak M(A)$. Using this isomorphism, the subspace of  real maximal ideals of a $\phi$-algebra $A$ (denoted by $\mathfrak R(A)$) has been defined, which consists of all points $p\in \mathfrak M(A)$ that every $a\in A$ is real-valued at $p$, see \cite[\S 4]{HJ}.

Now, if $\A$ is an algebra in a set $X$, then it is easy to see that $M(X, \A)$ is a $\phi$-algebra. Moreover, using Lemma \ref{multiple elements}, every ideal of $M(X, \A)$ is absolutely convex  and hence $(\beta_{\A}X, \tau_{\bar\A})= \mathfrak M(M(X, \A))$.    Moreover, by Corollary \ref{Theorem 8.4}, we conclude that $\mathfrak R(M(X, \A))=\upsilon_{\A}X$. Using these arguments, we have the following proposition.

\begin{proposition}\label{M=C(uX)}
Let $(X, \A)$ be a measurable space, then the following are equivalent.
\begin{enumerate}
\item $M(X, \A)\simeq C(Y)$, for some (topological) $P$-space $Y$.
\item $M(X, \A)\simeq C(\upsilon_{\A}X)$, by the mapping $f\to f^{\upsilon_{\A}}$.
\item $\bar\A_R=Z(\upsilon_{\A}X)$.
\end{enumerate}
\end{proposition}
\begin{proof}
(1) $\Leftrightarrow$ (2). This is due to \cite[1.4]{H1}, which states that $\mu_{\A} : M(X, \A) \to C(\upsilon_{\A}X)$ defined by $\mu_{\A}(f)=f^{\upsilon_{\A}}$ is  the only isomorphism between $M(X, \A)$ and any $ C(Y)$. Note that, $(\upsilon_{\A}X, \tau_{\bar\A_R})$ is a $P$-space, by Lemma \ref{Z(X)=B(X)} and the argument after Definition \ref{real ideal}.

(2) $\Leftrightarrow$ (3). It follows from (2), that $Z(\upsilon_{\A}X)=\bar \A_R$. On the other hand, the converse follows from Proposition \ref{beta_AX=beta X} and the argument after Corollary \ref{uX properties}.
\end{proof}
After \cite[Question 5.3]{ABS}, the authors give a measurable space, the interval $[0, 1]\subseteq \mathbb R$ equipped with the $\sigma$-algebra $\A$ of all Lebesgue measurable sets on $[0, 1]=:X$, to violate the possibility of $M(X, \A)\simeq C(X, \tau_{\A})$. In fact, using the previous result and since $(X, \A)$ is a realcompact measurable space   $M(X, \A)\simeq C(Y)$, for no topological space $Y$, note that $Z(X)=\mathcal P(X)\ne \A$.

However, for a measurable space $(X, \A)$, we   may have  $M(X, \A)$ is isomorphic to a $C(Y)$, for a $P$-space $Y$, but $M(X, \A)\ncong C(X, \tau_{\A})$, as the following example illustrate this.
\begin{example}
Let $X=\mathbb W$ be the set of all countable ordinals and $\A$ be the $\sigma$-algebra described in Example \ref{realcompact&realmeasurable}. As it pointed out in that example, $(X, \tau_{\A})$ is a discrete space and hence $Z(X)=\mathcal P(X)$. Now, it is easy to see that $|Z(X)|\ne |\A|$, thus $Z(X)\ncong\A$ and therefore $M(X, \A)\ncong C(X, \tau_{\A})$, by Proposition \ref{M=C(X)}.

On the other hand, using Example \ref{realcompact&realmeasurable}, $\upsilon_{\A}X= \mathbb W^*=W(\omega_1+1)$ and
$$\bar\A_R=\{A\subseteq X : |A|\le\aleph_0 \}\cup\{A\cup\{\omega_1\} : A \text{ contains a tail in }\mathbb W \}.$$
Therefore, $\tau_{\bar\A_R}=\mathcal P(X)\cup \{A\cup\{\omega_1\} : A \text{ contains a tail in }\mathbb W \}$. Using Lemma \ref{Z(X)=B(X)} and the argument after Definition \ref{real ideal}, $Z(\upsilon_{\A}X)=CO(\upsilon_{\A}X)$, the set of all clopen subsets of $\upsilon_{\A}X$. Now, let $B\in CO(\upsilon_{\A}X)$ be a clopen subset of $\upsilon_{\A}X$. If $\omega_1\in B$, then $B\cap X$ contains a tail in $\mathbb W$ as $B$ is open; otherwise, since $B$ is closed, $B\subseteq X$ and $|B|\le\aleph_0$. This shows that $Z(\upsilon_{\A}X)=\bar\A_R$ and hence $M(X, \A)\simeq C(\upsilon_{\A}X)$, by Proposition \ref{M=C(uX)}.
\end{example}

\section*{Acknowledgment}
 The authors would like to thank the referee for careful reading of this paper and for
helpful comments and suggestions.

\section*{Declarations}
On behalf of all authors, the corresponding author states that there is no conflict of interest.

\bibliographystyle{amsplain}

\end{document}